\documentclass[12pt]{article}
\usepackage{epsfig,psfrag,amsmath,amssymb,latexsym}
\usepackage{amscd }
\usepackage{amsfonts}
\usepackage{graphicx}
\newtheorem{theorem}{Theorem}[section]

\newtheorem{corollary}{Corollary}[section]

\newtheorem{example}{Example}[section]

\newtheorem{lemma}{Lemma}[section]

\newtheorem{proposition}{Proposition}[section]
\newtheorem{remark}{Remark}[section]

\newenvironment{proof}[1][Proof]{\textbf{#1.} }{\begin{flushright}\rule{0.5em}{0.5em}\end{flushright}}

\newcommand{\eop}{\vspace{-12pt}\begin{flushright}\rule{0.5em}{0.5em}\end{flushright}}

\numberwithin{equation}{section}

\def\X{{\cal X}}

\def\U{{\cal U}_n}
\def\V{{\cal V}_n}

\def\A{\mathbb{A}}
\def\e{\epsilon}
\def\v{{\rm Var}}
\def\S{{\cal S}_n}

\def\tZ{\tilde{Z}}
\begin{document}
\title{\large GENERAL APPROACH TO THE FLUCTUATIONS PROBLEM IN RANDOM SEQUENCE COMPARISON}
\author{J\"uri Lember\thanks{Supported by the Estonian Science Foundation
Grant nr.\ 9288 and targeted financing project SF0180015s12},
Heinrich Matzinger, Felipe Torres\thanks{corresponding author and research supported by the
DFG through the SFB 878 at University of M\"unster}}
\maketitle
\vspace{-0.5cm}
\hspace{0.6cm}\vbox{{\footnotesize
\noindent {\bf J\"uri Lember}, Tartu University, Institute of Mathematical Statistics\\
Liivi 2-513 50409, Tartu, Estonia. {\it E-mail:} jyril@ut.ee
\\
{\bf Heinrich Matzinger}, Georgia Tech, School of Mathematics\\
Atlanta, Georgia 30332-0160, U.S.A. {\it E-mail:} matzing@math.gatech.edu
\\
{\bf Felipe Torres}, M\"unster University, Institute for Mathematical Statistics \\
Einsteinstrasse 62, 48149 - M\"unster, Germany. {\it E-mail:} ftorrestapia@math.uni-muenster.de
}}

{\abstract
We present a general approach to the problem of determining the asymptotic order of the variance of the optimal score between two independent random sequences defined over an arbitrary finite alphabet. 
Our general approach is based on identifying random variables driving the fluctuations of the optimal score and conveniently choosing functions of them which exhibit certain monotonicity properties. We show how our general approach establishes a common theo\-re\-ti\-cal background for the techniques used by Matzinger {\it et al} in a series of previos articles \cite{BM,periodiclcs,HM,JJMM,goetze,FT1} studying the same problem in especial cases. Additionally, we explicitely apply our general approach to study the fluctuations of the optimal score between two random sequences over a finite alphabet (closing the study as initiated in \cite{goetze}) and of the length of the longest common subsequences between two random sequences with a certain block structure (generalizing part of \cite{FT1}).
}


\paragraph{AMS.} 60K35, 41A25, 60C05

\section{Introduction}

\subsection{Sequence comparison setting}\label{sec:pre}
Throughout this paper  $X=(X_1,X_2,\ldots X_n)$ and
$Y=(Y_1,Y_2,\ldots Y_n)$ are two  random strings, usually referred
as sequences, so that every random variable $X_i$ and $Y_i$ take
values on a finite alphabet $\mathbb{A}$. We shall assume that the
sequences $X$ and $Y$ have the same distribution and are
independent. The sample space of $X$ and $Y$ will be denoted by
$\X_n$. Clearly $\X_n\subseteq \A^n$ but, depending on the model,
the inclusion can be strict.
\\
The problem of measuring the similarity of $X$ and $Y$ is central in
many areas of applications including computational molecular biology
\cite{christianini, Durbin, Pevzner, SmithWaterman,
watermanintrocompbio} and computational linguistics
\cite{YangLi,LinOch,Melamed1,Melamed2}. In this paper, we adopt the
same notation as in \cite{LMT}, namely we consider a general
scoring scheme, where $S:\mathbb{A}\times \mathbb{A}\rightarrow
\mathbb{R}^+$ is a {\it pairwise scoring function} that  assigns a
score to each couple of letters from $\mathbb{A}$.
An {\it alignment} is a pair $(\rho,\tau)$ where
$\rho=(\rho_1,\rho_2,\ldots,\rho_k)$ and
$\mu=(\tau_1,\tau_2,\ldots,\tau_k)$ are two increasing sequences of
natural numbers, i.e. $1\leq \rho_1<\rho_2<...<\rho_k\leq n$ and $1\leq
\tau_1<\tau_2<\ldots<\tau_k\leq n.$ The integer $k$ is the number of
aligned letters, $n-k$ is the number of {\it gaps} in the alignment.
Note that our definition of gap slightly differs from the one that
is commonly used in the sequence alignment literature, where a gap
consists of maximal number of consecutive {\it indels} (insertion
and deletion) in one side. Our gap actually corresponds to a pair of
indels, one in $X$-side and another in $Y$-side. Since we consider
the sequences of equal length, to every indel in $X$-side
corresponds an indel in $Y$-side, so considering them pairwise is
justified. In other words, the number of gaps in our sense is the
number of indels in one sequence. We also consider a {\it gap price}
$\delta$.  Given the pairwise scoring function $S$ and the gap price
$\delta$, the score of the alignment $(\pi,\mu)$ when aligning $X$
and $Y$ is defined by
$$U_{(\rho,\tau)}(X,Y):=\sum_{i=1}^kS(X_{\rho_i},Y_{\tau_i})+ \delta (n-k).$$
In our general scoring scheme $\delta$ can also be positive,
although usually $\delta\leq 0$ penalizing the mismatch. For
negative  $\delta$, the quantity  $-\delta$ is usually called the
{\it gap penalty}.  The optimal alignment score of $X$ and $Y$ is
defined to be
\begin{equation}\label{Ln}
L_n:=L(X,Y):=\max_{(\pi,\mu)}U_{(\rho,\tau)}(X,Y),
\end{equation}
where the maximum above is taken over all possible alignments. To simplify the notation, in what follows, we shall denote
$Z:=(X,Y)$ so that $L_n=L(Z)$.
\\
\\
\noindent
When $\delta=0$ and the  scoring function assigns one to every pair
of similar letters and zero to all other pairs, i.e.
\begin{equation}\label{LCS-scoring}
S(a,b)=\left\{
           \begin{array}{ll}
             1, & \hbox{if $a=b$;} \\
             0, & \hbox{if $a\ne b$.}
           \end{array}
         \right.
\end{equation}
then $L(Z)$ is just the maximal number of aligned letters -- the
length of the {\it longest common subsequence (abbreviated by LCS)}. The longest
common subsequence is probably the most common measure of global
similarity between strings.
\subsection{The variance problem: history and the state of art}
Since $X,Y$ are random string, the optimal score $L_n$ is a random
variable. In order to distinguish related pairs of strings from
unrelated ones, it is relevant to study the distribution of $L_n$ for independent
sequences. When $X$ and $Y$ are take from an ergodic
processes then, by Kingman's subbaditive ergodic theorem, there
exists a constant $\gamma$ such that
\begin{equation}\label{sub}
{L_n\over n}\to \gamma \text{ a.s.\ and in $L_1$, as $n \to \infty$}.
\end{equation}
In the case of LCS, namely when $S$ is taken as in \eqref{LCS-scoring}, the constant $\gamma$ is sometimes called the {\it
Chvatal-Sankoff constant} and its value, although well estimated
(see \cite{Baeza1999,Boutet, BMNWilson,Steele86, Deken,Paterson1,Paterson2,Lueker,martinezlcs,meancurve,kiwi}) remains unknown even for
as simple cases as i.i.d.\ Bernoulli sequences. The existence of
$\gamma$ was first noticed by Chvatal and Sankoff in their
pioneering paper \cite{Sankoff1}, where they proved that the limit
\begin{equation}\label{el}
\gamma:=\lim_{n\rightarrow\infty}\frac{E L_n}{n}
\end{equation}
 exists. In \cite{Alexander1} the rate of the convergence in
 (\ref{el}) was for the first time established, and in \cite{LMT} the authors improved the previous results introducing a new technique based on entropy and combinatorics, which gives a little more bout the path structure of the optimal alignments.
\paragraph{The fluctuations of $L_n$.} To make inferences on $L_n$, besides the convergence
(\ref{sub}),  the size of the variance $\v [L_n]$ is essential.
Unfortunately, not much is known about $\v [L_n]$ and its
asymptotic order is one  of the central open problems in string matching theory.\\
Monte-Carlo simulations lead Chvatal and Sankoff  in \cite{Sankoff1} to  conjecture that, in the case of LCS, $\v[L_n]=o(n^{2\over 3})$ for i.i.d.\ $\frac{1}{2}$-Bernoulli sequences. Using an Efron-Stein type of inequality,
Steele \cite{Steele86} proved that there exist a constant  $B<\infty $
such that $\v[L_n]\leq B n$. In \cite{Waterman-estimation}, and always in the LCS case, Waterman asks
whether this linear bound can be improved. His
simulations show that this is not the case and $\v[L_n]$ should grow
linearly. Still in the LCS case, Boutet de Monvel \cite{Boutet} interprets his simulations the same way. In a series of papers containing different settings, Matzinger {\it et al.}\ have been investigating the
asymptotic behavior of $\v [L_n]$. Their goal is to
find out whether there exists a constant $b>0$ (not depending on
$n$) such that $\v[L_n]\geq  bn$. Together with Steele's bound, this
means that $bn\leq \v[L_n] \leq Bn$, i.e. $\v[L_n]=\Theta(n)$ (we say that a
sequence $a_n$ is of order $\Theta(n)$ if, there exist constants
$0<b<B<\infty$ such that $bn\leq a_n\leq Bn$ for all $n$ large enough). So far, most
of the research to show that $\v[L_n]=\Theta(n)$ has been done in the case of LCS:

\begin{itemize}
\item In \cite{periodiclcs}, $X$ is a $\frac{1}{2}$-Bernoulli binary sequence and $Y$ is a non-random periodic binary sequence,
\item In \cite{BM}, $X$ is a $\frac{1}{2}$-Bernoulli binary sequence and $Y$ is an i.i.d. random sequence over a 3 -- symbols alphabet,
\item In \cite{HM}, both $X$ and $Y$ are $\frac{1}{2}$-Bernoulli binary sequences but they are aligned by using a score function which gives more weight when aligning ones than aligning zeros,
\item In \cite{JJMM}, both $X$ and $Y$ are i.i.d.\ binary sequences, but one symbol has much smaller probability than the other. That is a so called {\it case of low entropy}.
\item In \cite{FT1}, both $X$ and $Y$ are binary sequence having a multinomial block structure. That is, for the first time, a so called {\it case of high entropy}.
\end{itemize}

\noindent
Another related string matching problem is the so called {\it longest increasing subsequence (abbreviated by LIS) problem}. Given a sequence $X$, to find the LIS of $X$ is to find an increasing sequence of natural numbers $1 \le i_1 \le i_2 \le \cdots \le i_k \le n$ such that $X_{i_1} \le X_{i_2} \le \cdots \le X_{i_k}$. The LIS problem can be seen as a particular case of the LCS problem, in the following way: Let $X:=1\, 2 \cdots\, n$ be the sequence of the first $n$ increasing integers and let $\sigma(X):=\sigma(1)\sigma(2)\cdots \sigma(n)$ be the sequence of its random permutation. Then, a LIS of $X$ corresponds to a LCS of $X$ and $\sigma(X)$. Due to this equivalence, it was thought that the LIS and the LCS have fluctuations of the same order, which now we know it is not true. In this direction Houdre, Lember and Matzinger \cite{HLM} studied an hydrid problem, namely the fluctuations of $\ell_n$ defined as the length of the longest common increasing subsequence of $X$ and $Y$, where $X$ and $Y$ are i.i.d.\ $\frac{1}{2}$-Bernoulli binary sequences, and a longest common increasing subsequence of $X$ and $Y$ is just a LIS of $X$ and of $Y$ simultaneously. They showed that $n^{-{1/2}}(\ell_n-E\ell_n)$ converges in law to a functional of two Brownian motions, which implies that $\v[\ell_n]=\Theta(n)$ holds as well. There is also a connection between the LCS of two random sequences and a certain passage percolation problem \cite{Alexander1}.
\\
For the case of general scoring, to our best knowlodge, the only previous partial results on fluctuations are cointained in \cite{goetze}.
\subsection{Main results and the organization of the paper}
Recall that we aim to prove (or disprove) the order of the variance
$\v[L_n]=\Theta(n)$ and due to Steele's upper bound it suffices to
prove (or disprove) the existence of $b>0$ so that $\v[L_n]\geq bn$.
All available proofs (by Matzinger {\it et al.})\ of the existence of such $b$ follow more or less the
same philosophy and can be split into two parts, strategy that we call {\it two-step approach}. 
The first part of this approach is to find a random mapping independent of $Z$, usually called a {\it random
transformation},
$$\mathcal{R}: \X_n\times\X_n \to \X_n\times\X_n$$
such that, for most of the outcomes $z\in
\X_n\times\X_n$ of $Z$, increases the score at least by some fixed
amount $\e_o>0$. More precisely, the random transformation should be that 
for some $\alpha>0$ there exists a set $B_n\subset \X_n\times\X_n$ having
probability at least $1-\exp[-n^{\alpha}]$ so that, for every $z\in
B_n$, the expected score of $\mathcal{R}(z)$ exceeds the
score of $z$ by $\e_o$ (where the expectation is taken over the
randomness involved in the transformation), namely
$$E\big[L(\mathcal{R}(z))\big]\geq L(z)+\e_o.$$
Before stating this requirement formally, let us introduce a
useful notation: $\tZ:=\mathcal{R}(Z)$. Thus $\tZ$ is obtained from $Z$
by applying a random modification to $Z$ and the additional
randomness is independent of $Z$. Formally, the first step of the approach
is to find a random transformation so that for some universal
constants $\alpha>0$ and $\e_o>0$, the following inequality holds:
\begin{equation}\label{main}
P\big(E[L(\tZ)-L(Z)|Z]\geq \e_o\big)\geq 1-\exp[-n^{\alpha}].
\end{equation}
Besides (\ref{main}), the random transformation has to satisfy some
other requirements. This other requirements and their influence 
on the fluctuations of $L_n$ form the second step of the two-step approach. 
Roughly speaking, there should also exist an associated
function of  $Z$, let us call it $\mathfrak{u}(Z)$, so that applying the random
transformation to $Z$ increases the value of $\mathfrak{u}$. The variance of
$L_n$ can be then lower bounded by the variance of $U:=\mathfrak{u}(Z)$ so that
the constant $b$ exists if the variance of $\mathfrak{u}(Z)$ is linear on $n$.
 This second step is formally presented and explained with details in
Subsection \ref{rt}, where
 we also briefly introduce the random
transformation and the random variable $U$ used so far. Let us remark
that in earlier articles of the subject \cite{BM,HM,JJMM,FT1}, the random transformation is not explicitly defined, but the
variance driving  random variable $U$ is always there, and one can
easily define the random transformation as well. Let us also  mention that to show
(\ref{main}) for a suitable chosen transformation is not an easy
task and, typically, needs a lot of effort.
The second step of the approach consists of showing that (\ref{main})
implies the existence of $b$. This proof depends on the model,
on the chosen transformation and its components (vectors $U$ and $V$, see
Subsection \ref{rt}). One of the goals of the present paper is to
present a general setup and a general proof for the second step. With
such a general proof in hand, all the the future proofs of  the
existence of $b$ could be remarkably shortened and simplified. Our
general approach is based on the same strategy as in \cite{JJMM, FT1}, 
but remarkably shorter and simpler (see
also Remark 6 before the proof of Theorem \ref{gen2}). These general
results are Theorem \ref{gen1} and Theorem \ref{gen2}, both presented in
Section \ref{sec:ths}. 
\\
In order to see in action our two-step approach, 
we include two applications which bring us up two new fluctuation results:
\paragraph{First application: optimal score of i.i.d.\
sequences.} In Section \ref{sec:scor}, $X$ and $Y$  are
$\A$-valued i.i.d. random variables, being compared under the general scoring
scheme described in Subsection \ref{sec:pre}. In this case, the
random transformation consists of uniformly choosing a specific letter $a\in
\A$ and turning it into another specific
letter $b\in \A$. In \cite{goetze}, it has been proven that when the
gap price is relatively low and the scoring function $S$ satisfies
some mild asymmetry assumptions, then the described transformation
satisfies (\ref{main}), see Theorem \ref{app1} and the remarks after
it. Thus, for sufficiently low gap price $\delta$, the first step
holds true. In Section \ref{sec:scor}, we show that
all other assumptions of Theorem \ref{gen1} and Theorem \ref{gen2} are
fulfilled, so that the second step holds true and thus there exists the 
desired constant $b$  (Theorem \ref{main2}). Hence, Section \ref{sec:scor} completes the study
started in \cite{goetze} and, to our best knowledge, we obtain the first result where the order of variance
$\v[L_n]=\Theta(n)$ is proven in a setup other than LCS
of binary sequences. It is important to note that for the  second
step (Theorem \ref{main2}), no assumption on $\delta$ nor on the
scoring function $S$ are needed. Hence, whenever the assumptions in
the first step can be relaxed (i.e. generalization of Theorem \ref{app1}),
the second step still holds true and the order of variance
$\v[L_n]=\Theta(n)$   can be automatically deduced.
\paragraph{Second application: The length of the LCS of random i.i.d.\ block sequences.}
Unfortunately, the current assumption on the gap price $\delta$
makes Theorem \ref{app1} not applicable to the length of the LCS of two
independent i.i.d. sequences, thus this case (except the special model in
\cite{JJMM}) is still open. In order to approach this still open question from another
point of view, in Section \ref{sec:LCS} we consider $X$ and
$Y$ not to be i.i.d.\ sequences any more, but we keep $\A=\{0,1\}$ consisting of two colours
(i.e.\ the sequences are still binary) and the scoring function to be
(\ref{LCS-scoring}), also $\delta=0$. Hence, $L_n$ is the length of
the LCS. The difference from the setup considered in Section \ref{sec:scor} lies
in the random structure of $X$ and $Y$. Let us briefly explain the model.
Note that any binary sequence can be considered as a
concatenation of {\it blocks} with switching colors (from $0$ to $1$ or viceversa). Here a block is
merely a subword of the sequence having all letters of the same
color and a different color before and after it. Hence, every binary
sequence  is fully determined by the lengths of its blocks and the
color of its first block. Therefore, every  infinite i.i.d.\
Bernoulli sequence $X$ with parameter ${1\over 2}$ can be considered as
an i.i.d sequence of blocks whose lengths are geometrically distributed, where
the first block has colour either $0$ or $1$  with probability ${1\over 2}$.
$X$ can be, in a sense, approximated by a (binary) random
sequence $\hat{X}$ with finite range of possible block lengths.
 Indeed, the probability of finding a very long block in $X$ is very small,
hence such an approximation of $X$ by $\hat{X}$ is justified (note that although the
blocks remain to be i.i.d, $\hat{X}$ is not an
i.i.d.\ sequence any more). This is the situation in Section
\ref{sec:LCS}: instead of considering $X$ (and $Y$) as the first $n$
elements of an i.i.d.\ infinite Bernoulli sequence with parameter
${1\over 2}$, we take them as the first $n$ elements of an
infinite sequence $X_1,X_2,\ldots$ obtained by i.i.d.\ concatenating blocks of alternated colours
of random lengths distributed on $\{l-1,l,l+1\}$, where $l>2$ is a fixed integer. 
For a formal description of this block model see Subsection \ref{sec:block-model}. The
restriction that the block lengths can only have three possible
values is made in order to have a simplified exposition of the technique. We believe that
the results in Section \ref{sec:LCS} also hold for any finite range of
possible block lengths. \\
Considering such a block model is motivated by the following arguments.
First, it is a common practice in random sequence comparison to
approximate a target model (i.i.d. Bernoulli sequences, in our case)
by some  more tractable model. In random sequence comparison, the
more tractable model typically has lower entropy. Secondly, as it is
shown in \cite{FT1}, for the case where all three possible
block lengths have equal probability, there exists a random
transformation so that (\ref{main}) holds, see Lemma \ref{phdfirststep}.
The random transformation in this case -- let us call it the {\it
block-transformation} -- is the following: pick uniformly an
arbitrary block of $X$ with length $l-1$ and independently an arbitrary block of
$X$ with length $l+1$. Then, change
them both into blocks of length $l$. Thus, in this particular case
where all block lengths have equal probability, the first-step is
accomplished. In Section \ref{sec:LCS}, we show that the
block-transformation and corresponding random variables satisfy all
other assumptions of Theorem \ref{gen1} and Theorem \ref{gen2}, so that the
existence of $b$ follows (Theorem \ref{LCSthm}). Thus, for the case
of equiprobable  block lengths, the order of variance
$\v[L_n]=\Theta(n)$ has now been proved. Since the uniform distribution of
of block lengths was not used in the second step (see Theorem
\ref{LCSthm}), it follows that the same order of variance
automatically holds if an equivalent to Lemma \ref{phdfirststep} 
without the uniform distribution assumption can be shown. Again, we believe that such
a generalization is true.
\section{The two-step approach}\label{sec:ths}
\subsection{Preliminaries}
\begin{proposition}\label{prop} Let $N$ be an integer-valued random variable
taking values on interval $I$. Let $f: I\to \mathbb{R}$ be a
monotone function so that for a $c>0$,  $f(k)-f(k-1)\geq c$ (or
$f(k-1)-f(k)\geq c$) for every $k, k-1\in I$. Then $\v[f(N)]\geq
c^2\v[N]$.\end{proposition}
For a proof of this statement see \cite{BM}. The next corollary replaces the more involved Lemma 3.3 in \cite{BM} or Lemma 5.0.3 in \cite{FT1}. In our general approach, we need a simpler version because we use the decomposition \eqref{deco} (see Remark 6. after Theorem \ref{gen1}):
\begin{corollary}\label{popcor2}  Let $N$ be an integer-valued random variable
taking values on the set $\mathcal{Z}:=\{z_1,z_2,\ldots\}\subset \mathbb{Z}$,
. Let $k_o:=\sup_{i\geq 2}(z_i-z_{i-1})<\infty.$  Let $f$ be an
increasing function defined on $\mathcal{Z}$ so that for $\delta>0$ it holds
$$f(z_i)-f(z_{i-1})\geq \delta,\quad \forall i\geq 2.$$
Then
 $$\v[f(N)]\geq {\delta^2\over k_o^2}\v[N].$$
\end{corollary}
\begin{proof} Let $M$ be a random variable taking values on the set
$I=\{1,2,\ldots\}$ defined as follows: $M=i$ iff $N=z_i$. Let $g$ be
an increasing function  on $I$ defined as follows: $g(i):=f(z_i)$.
Thus $g(i+1)-g(i)\geq \delta$ and $g(M)=f(N)$. By Proposition \ref{prop}
$$\v[f(N)]=\v[g(M)]\geq \delta^2 \v[M]\geq  \big({\delta \over
k_o}\big)^2 \v[N],$$ where the last inequality follows from the
inequality $\v[N]\leq k_o^2 \v[M].$\end{proof}
\begin{lemma}[Chebychev's inequality]
Let $U$ be a random variable, then for any constant $\zeta>0$ we have
\begin{equation}\label{chevi}
P\left(\,\,|U-\mathrm{E}[U]| \ge \zeta \sqrt{{\rm VAR}[U]}\,\,\right) \le \frac{1}{\zeta^2}.
\end{equation}
\end{lemma}
\begin{lemma} [H\"offding's inequality]
\label{hoeffdingco} Let $a>0$ be constant and $V_1,V_2,\dots$ be an
i.i.d sequence of bounded random variables such that:
\[ {\rm P}(|V_i-{\rm E}[V_i]| \le a)=1\]
for every $i=1,2,\dots$ Then for every $\Delta>0$, we have that:
\begin{equation}\label{corAH}
{\rm P}\left(\,\left|\frac{V_1+\dots+V_n}{n}-{\rm E}[V_1] \right|\ge \Delta\right) \le 2\exp\left( -\frac{\Delta^2}{2a^2}\cdot n\right)
\end{equation}
\end{lemma}
The following lemma follows from the local central limit theorem (section 2.5 in \cite{Durret}):
\begin{lemma}\label{binomlclt}
Let $X\sim B(m,p)$ be a binomial   random variable  with parameters
$m$ and $p$. Then, for any constant $\beta
>0$, there exists $b(\beta)$ and $m_0(\beta)$ so that for every
$m>m_o$ and $$ i \in [mp-\beta \sqrt{m},mp+\beta \sqrt{m}]=:I_m,$$
it holds
\begin{align}\label{binom}
P(X=i)&={m \choose i }p^i(1-p)^{m-i}\geq {1\over b\sqrt{ m}}.
\end{align}
Moreover, there exists an universal constant $c_1(\beta)>0$ and
$m_1(\beta)$ so that for every $m>m_1$
\begin{equation}\label{binomvar}
\v [X|X\in I_m]\geq  c_1 m.
\end{equation}
\end{lemma}
Applying Lemma \ref{binom} repeatedly on marginals, we obtain a
multinomial  corollary:
\begin{corollary}\label{multiclt}
Let $(X,Y,Z)$ be a multinomial random vector with parameters $m$ and
$p_1$, $p_2$, $p_3$ such that $p_1+p_2+p_3=1$. Then, for any constant  $\beta
>0$, there exists $b(\beta)$ and $m_0(\beta)$ so that  for every $m>m_o$ and  $$\
(i,j)\in [mp_2-\beta \sqrt{m},mp_2+\beta \sqrt{m}]\times
[mp_1-\beta\sqrt{m},mp_1+\beta\sqrt{m}],$$ it holds
\begin{align}\label{multi}
P(X=i,Y=j)&={m \choose i \; j \; (m-i-j)}p_1^ip^j_2p_3^{m-i-j}\geq {1\over b m}.
\end{align}
\end{corollary}
\subsection{General fluctuations results}
Let ${\cal X}_n$ be the sample space of $X$ and $Y$ so that ${\cal
X}_n\times {\cal X}_n$ is the sample space of $Z:=(X,Y)$. In the
following, we are considering the functions
$$\mathfrak{u}:\quad {\cal X}_n\times {\cal X}_n \to \mathbb{Z},\quad \mathfrak{v}:\quad {\cal X}_n\times {\cal X}_n \to
\mathbb{Z}^d$$ so that $U:=\mathfrak{u}(Z)$ (resp. $V:=\mathfrak{v}(Z)$ ) is an integer
values random variable (resp. vector). We shall denote by $\S$,
$\S^U$ and $\S^V$ the support of $(U,V)$, $U$ and $V$, respectively.
Hence $\S\subset \mathbb{Z}^{d+1}$, $\S^U\subset \mathbb{Z}$ and
$\S^V\subset \mathbb{Z}^{d}$. For every $v\in \S^V$, we define the
fiber of $\S^U$ as follows
$$\S(v):=\{u\in \S^U: (u,v)\in \S\}.$$
For any $(u,v)\in \S$, let
$$l(u,v):=E[L(Z)|U=u,V=v].$$
In what follows, we shall often consider the conditional
distribution of $U$ given that $V$ takes a particular value $v$.
Therefore, we shall denote by $U_{(v)}$  a random variable that has this conditional distribution, i.e. for any
$z\in \mathbb{Z}$, it holds
$$P(U_{(v)}=z)=P(U=z|V=v).$$
We shall also consider the sets of "typical values" of $V$ and
$U_{(v)}$. More precisely, we shall define the sets $\V\subset \S^V$
that contain (in some sense) the values of $V$ that are of our
interest. Similarly, for every $v\in \V$, we shall define the sets
$\U(v)$ that (again in some sense) contains the values of
$U_{(v)}$ that are of our interest. Roughly speaking, in what follows we shall always
condition on the events $\{V\in \V\}$ and $\{U_{(v)}\in \U(v)\}$.
\begin{theorem}\label{gen1} Assume the existence of sets $\V\subset
\S^V$ and $\U(v)\subset \S(v)$, for $v\in \V$, so that for some constants
$\delta>0$ and $k_o\in \mathbb{N}$, the following conditions hold:
\begin{description}
  \item[1)] For every $v\in \V$ and $u_1,u_2\in \U(v)$ such that $u_1<u_2$, it holds
\begin{equation}\label{main1}
l(u_2,v)-l(u_1,v)\geq \delta.
\end{equation}
  \item[2)] There exists $\psi_n$ so that for every $v\in \V$, the
following lower bound holds
\begin{equation}\label{phi}
\v[U_{(v)}|U_{(v)}\in \U(v)]\geq \psi_n.
\end{equation}
  \item[3)] There exists $k_o>\infty$ so that for every $v\in \V$
and $u_1\in \U(v)$, there exists an $u_2  \in \U(v)$ so that
$|u_1-u_2|\leq u_1+k_o$.
\end{description}
Then
\begin{equation}\label{claim}
\v[L(Z)]\geq \left({\delta\over k_o}\right)^2\cdot \psi_n \cdot \sum_{v\in \V}P(U_{(v)}\in \U(v))P(V=v).
\end{equation}
\end{theorem}
\begin{proof}
It is clear that
\begin{equation}\label{deco}
\v[L(Z)]=E\big(\v[L(Z)|U,V]\big)+\v\big(E[L(Z)|U,V]\big)\geq
\v\big(l(U,V)\big). \end{equation}
 We aim to bound $\v\big(l(U,V)\big)$ from below.
We condition on $V$ and use the same formula to get
\begin{align}\nonumber
\v\big(l(U,V)\big)&=E\big(\v[l(U,V)|V]\big)+\v\big(E[l(U,V)|V]\big)\geq
E\big(\v[l(U,V)|V]\big)\\
\nonumber
&=\sum_{v\in \S^V}\v[l(U,v)|V=v]P(V=v)\geq \sum_{v\in \V}\v[l(U,v)|V=v]P(V=v)
\\\label{summ}
&=\sum_{v\in \V}\v[l(U_{(v)},v)]P(V=v).
\end{align}
Conditioning on the event $\{U_{(v)}\in \U\}$, we see that
$$\v[l(U_{(v)},v)]\geq \v[l(U_{(v)},v)|U_{(v)}\in \U(v)]P(U_{(v)}\in \U).$$
By assumption {\bf 1)}, on the set $\U(v)$ the function $l$
satisfies (\ref{main1}). By assumption {\bf 3)}, the two consecutive
atoms of $\S(v)\cap \U$ are at most $k_o$ apart from each other. By
Corollary \ref{popcor2}, thus
\begin{equation}\label{oi2}
\v[l(U_{(v)},v)|U_{(v)}\in \U]\geq {\delta^2\over k_o^2} \cdot
\v[U_{(v)}|U_{(v)}\in \U].\end{equation} Thus (\ref{summ}) can lower bounded by
\begin{align*}
\sum_{v\in \V}\v[l(U_{(v)},v)]P(V=v)&\geq
\left({\delta\over k_o}\right)^2 \cdot \sum_{v\in \V} \v[U_{(v)}|U_{(v)}\in \U(v)]P\big(U_{(v)}\in \U(v)\big)P(V=v)\\
&\geq \left({\delta\over k_o}\right)^2 \cdot \psi_n \cdot \sum_{v\in \V}  P\big(U_{(v)}\in \U(v)\big)P(V=v).
\end{align*}
\end{proof}
\paragraph{Remarks:} 
\begin{enumerate}
                 \item The theorem above is non-asymptotic. It means
that $n$ is fixed and, therefore, could be removed from the
statement. However, writing the theorem in such a way, we try to
stress out that $\delta$ and $k_o$ should be independent of $n$ when
applying the theorem. Obviously $X,Y,Z,U,V$ will depend on $n$ too,
but we do not explicitely include that in the notation.
                  \item In order to get a linear lower bound from (\ref{claim}),
it suffices to show that for some constant $b>0$ it holds,
$$\psi_n \cdot \sum_{v\in \V}P(U_{(v)}\in \U(v))P(V=v)\geq b \,n$$
Typically $\psi_n$ is linear on $n$ so that for a constant $d>0$ we will have
$\psi_n\geq dn$, and the sets $\U(v)$ and $\V$ are such that for
constants $d_1$ and $d_2$ it holds,
\begin{equation}\label{dd1}
P(V\in \V)\geq {d_1},\quad P(U_{(v)}\in \U(v))\geq {d_2},\quad
\forall v\in \V.
\end{equation}
Then the right side of (\ref{claim}) has a linear lower bound as
desired:
$$
\psi_n \cdot \sum_{v\in \V}P(U_{(v)}\in \U(v))P(V=v)\geq (d_1\, d_2\, d\,)\, n.$$
                 \item The most crucial assumption of Theorem \ref{gen1} is
assumption {\bf 1)}. It states that the function $u\mapsto l(u,v)$ increases at
least by certain amount $\delta$ on the set where $U$ and $V$ take
their typical values. The core of the approach is to find $U$ and $V$ such that (\ref{main1})
holds. Later in concrete settings, we shall see how
(\ref{main1}) is achieved in practice.
               \item The condition {\bf 2)} barely states the
existence of an uniform lower bound for the
conditional variance (i.e. independent of $v$). Some trivial bounds clearly exist, but as
explained above, $\psi_n$ has to grow linearly in order to get a linear lower bound for $\v[L_n]$.
               \item The condition {\bf 3)} is of technical nature.
In particular, it holds if $\U(v)$ is a lattice of span $k_o$, i.e.
for  integers $m$ and $u_o$
\begin{equation}\label{lattice}
\U(v)=\{u_o+k_o,u_o+2k_o,\ldots,u_o+ m k_o\}.
\end{equation}
 As we shall see, this is a
typical situation in practice.
\item The proof is based on the decomposition (\ref{deco}). In all
previous papers, the lower bound of $\v[L(Z)]$ was achieved by
bounding below the (expectation) of conditional variance
$\v[L(Z)|U,V]$ (\cite{periodiclcs,BM,HM,JJMM,FT1}). This approach often involves martingale's arguments
(via H\"offding-Azuma inequality), non-trivial combinatorial estimates and a generalization of
Proposition \ref{prop}. In this paper, however, we bound the
variance of conditional expectation $\v\big(E[L(Z)|U,V]\big)$.
Although the main idea remains the same, the proof is now much
shorter and less technical, relying solely on Proposition
\ref{prop}.
\end{enumerate}
\begin{corollary}\label{cor1} Let, for any $v\in \V$, the set $\U(v)$ be defined as $\U(v)=\U\cap
\S(v)$, where $\U\subset \mathbb{R}$ is a subset independent of $v$.
If the assumptions of Theorem \ref{gen1} are satisfied, then
\begin{equation}\label{claimU}
\v[L(Z)]\geq \left({\delta\over k_o}\right)^2\cdot \psi_n \cdot P(U\in
\U,V\in \V).
\end{equation}
\end{corollary}
\begin{proof}
The (\ref{claimU}) follows from (\ref{claim}):
\begin{align*}
\sum_{v\in \V}  P\big(U_{(v)}\in \U(v)\big)P(V=v)  & = \sum_{v\in
\V} P\big(U_{(v)}\in \U \big)P(V=v) =\sum_{v\in \V} P\big(U\in \U
|V=v  \big)P(V=v)\\
&=\sum_{v\in \V}  P\big(U\in \U, V=v)=P(U\in \U,V\in \V).
\end{align*}
\end{proof}
\subsection{Random transformation and the condition (\ref{main1})}\label{rt} 
In order to simplify the notation, we consider the case where $\U$ is
an integer interval and that, for any $v\in \V$, the fiber $\S(v)$ is a
lattice with span $k_o\geq 1$. Thus, for every $v\in \V$ there
exists an integer $m$ (depending on $n$ and $v$) so that $\S(v)\cap
\U$ is as in (\ref{lattice}). As explained in Remark 5, in this case
the condition {\bf 3)} of Theorem \ref{gen1} is fulfilled.\\\\
For any $(u,v)\in \S$, let $P_{(u,v)}$ denote the law of of $Z$
given $U=u$ and $V=v$. Thus
$$P_{(u,v)}(z)=P(Z=z|U=u,V=v).$$
Recall from the introduction that the core of the whole two-step
approach is the existence of a random transformation
$\mathcal{R}: \X_n\times\X_n \to \X_n\times\X_n$ independent of $Z$ that
satisfies (\ref{main}). In order to make this approach to work, the
transformation should be associated with the $U$ and $V$ in the
following way: for a typical  $z\in \X_n\times\X_n$, the
transformation increases $\mathfrak{u}(z)$ by $k_o$ unit and leaves $\mathfrak{v}(z)$
unchanged. Typically there are many such (random or non-random)
mappings, but to ensure (\ref{main1}), the transformation should be
chosen so that some additional assumptions are fulfilled. Recall
that $\tZ$ is obtained from $Z$ by applying a random modification to
$Z$ and the additional randomness is independent of $Z$. As
mentioned above, the transformation increases $U=\mathfrak{u}(Z)$ by $k_o$ and
leaves $V=\mathfrak{v}(Z)$ unchanged, thus (at least for  the typical values of
$Z$), it holds
\begin{equation}\label{shift}
\mathfrak{u}(\tZ)=\mathfrak{u}(Z)+k_o,\quad \mathfrak{v}(\tZ)=\mathfrak{v}(Z).
\end{equation}
In addition, we need the following assumptions to be true:
\begin{description}
  \item[A1] There exist universal (not depending on $n$) constants  $\alpha>0$ and $\e_o>0$ such that
\begin{equation*}
P\big(E[L(\tZ)-L(Z)|Z]\geq \e_o\big)\geq 1-\exp[-n^{\alpha}].
\end{equation*}
  \item[A2] There exists universal constant  $A<\infty$ so that $L(\tZ)-L(Z)\geq -A$.
  \item[A3] For any $(u,v)$ such that $v\in \V$ and  $u\in \U(v)$,   the following implication holds:
\begin{equation}\label{impli}
\text{If  }Z\sim P_{(u,v)},\text{   then   }\tZ \sim P_{(u+k_o,v)}.
\end{equation}
\end{description}
{\bf Remarks:} \begin{enumerate}
                 \item The assumption {\bf A1} is the condition
(\ref{main}) explained already in  Introduction.
                 \item The assumption {\bf A2} states that by applying the random transformation, the
maximum decrease of the score is at most $A$. This assumption
usually holds for trivial reasons.
                 \item Note that (\ref{impli}) implies
(\ref{shift}).  However, the condition (\ref{impli}) is more
restrictive and (except some trivial cases) to achieve it, the
transformation $R$ has to be random.
\item If $\U(v)=\U\cap \S(v)$, then $v\in \V$ and $u\in \U(v)$ holds
if and only if $(u,v)\in \S\cap (\U \times \V)$.
               \end{enumerate}
\begin{theorem}\label{gen2} Assume the existence of a random transformation so
that for every $n$, {\bf A1}, {\bf A2} and {\bf A3} hold. Suppose
that there exists a constant $a>0$ so that for any $(u,v)$ such that
$v\in \V$ and $u\in \U(v)$, it holds
\begin{equation}\label{LCLT}
P(U=u,V=v)\geq {1\over a n}.
\end{equation}
Then there exists a $n_5<\infty$ so that for every $n>n_5$ the
assumption {\bf 1)} of Theorem \ref{gen1} is fulfilled with
$\delta={\e_o\over 2}$.\end{theorem}
\begin{proof} Let $v \in \V$ and $u\in \U(v)$. Let $Z_{(u,v)}$ be a
random vector having the distribution $P_{(u,v)}$. By {\bf A3},
thus,
$$l(u+k_o,v)=E[L(\tZ_{(u,v)})].$$
Hence
\begin{align*}
l(u+k_o,v)-l(u,v)&=E[L(\tilde{Z}_{(u,v)})]-E[L({Z}_{(u,v)})]=E[L(\tilde{Z}_{(u,v)})-L({Z}_{(u,v)})]\\
&=E\big(E[L(\tilde{Z}_{(u,v)})-L({Z}_{(u,v)})\big|Z_{(u,v)}]\big).\end{align*}
Let $B_n \subset \X_n\times \X_n$ be the set of outcomes of $Z$ such
that
$$\big\{E[L(\tZ)-L(Z)|Z]\geq \epsilon_o\}=\{Z\in
B_n\}.$$ By assumption {\bf A2}, for any pair of sequences $z$, the
worst decrease of the score, when applying the block-transformation
is $- A$. Hence
\begin{equation*}
E\big(E[L(\tZ_{(u,v)})-L({Z}_{(u,v)})\big|{Z}_{(u,v)}]\big)\geq \e
P\big({Z}_{(u,v)} \in B_n(\e)\big)-A P({Z}_{(u,v)}\not \in
B_n(\e)).\end{equation*} By {\bf A1},  $P\big(Z\not \in B_n\big)\leq
\exp[-n^{\alpha}]$. Therefore
$$P\big({Z}_{(u,v)}\not\in B_n\big)=P\big(Z\not \in
B_n|U=u,V=v\big)\leq {P\big(Z\not \in B_n\big)\over P(U= u,V=v)}\leq
a n \exp[-n^{\alpha}]$$ where the last inequality follows from
(\ref{LCLT}). Take now $ n_5$ so big that for any $n>n_5$, we have
$$\e_o \big(1- {a n  }\exp[-n^{\alpha}] \big)-A{ a n
}\exp[-n^{\alpha}]>{\e_o \over 2}.$$ Hence, for any $n>n_5$ and for
any $(u,v)$ such that $v\in \V$ and $u\in  \U(v)$, we have
\begin{align} l(u+k_o,v)-l(u,v)\geq \e_o\big(1- an \exp[-n^{\alpha}]\big)-Aan \exp[-n^{\alpha}]\geq {\e_o\over
2}.\end{align}\end{proof}

\subsection{Covered previous results}
Before turning into new results presented in the
subsequent sections, let us briefly mention how the random
transformation as well as the associated random variables were defined
in already obtained results:
\begin{itemize}
 \item In \cite{periodiclcs}, the random variable $U$ is the number of
matching replica points while $V$ is a constant. Roughly speaking, a letter $X_i$ is a
replica point if it has a neighborhood that matches exactly with the
periodic sequence (i.e.\ it has the same periodic pattern). The replica point
itself can or cannot match, and it is shown that the number of
matching replica points has variance proportional to $n$. In
\cite{periodiclcs} the random transformation is not explicitly defined,
but one can take it as uniformly choosing a replica point with
prescribed color and change its value.

\item In \cite{BM} the random sequence $X$ is built up on the alphabet $\{0,1,a\}$. The random variable $U$ is the number of $a$'s in $X$, while $V$ is a constant. The random transformation is hidden in the so called {\it drop-scheme of random bits}, used to construct the sequence $X^{01}$ which is the subsequence of $X$ only having $0$'s and $1$'s. Roughly speaking, the drop-scheme of random bits consists on, starting from a binary random sequence of length two, to flip a coin and to add the resulting symbol into the previous sequence at an uniformly chosen location, so increasing the length, until reaching a length $n-U$.

\item In \cite{HM} the scoring function is such that $S(1,0)=S(0,1)=0, S(0,0)=1$ and $S(1,1) \in \mathbb{R}$. The random transformation consists, in $X$, to uniformly choose a block of length five, to take one of its symbols out and to add it to a uniformly chosen block of length one. The random variable $U$ is the number of blocks of length two and of length four, and $V$ is the number of blocks of the other lengths.

 \item In \cite{JJMM}, both sequences typically consist of
many zeros and few ones. The random transformation, uniformly at random picks an
arbitrary one in $Z$ and turns it into a zero. The variable $U$ is the number of ones in $Z$, $V$ is a
constant. Hence, this case is essentially the same as considered in
the Section \ref{sec:scor} with $|\A|=2$ and Theorem \ref{main2} nicely
generalizes Theorem 2.1 in \cite{JJMM}.

\item In \cite{FT1}, the random transformation and the random variables $(U,V)$ are defined to be as in Section \ref{sec:LCS}.

\end{itemize}
\section{Optimal score of random i.i.d.\ sequences}\label{sec:scor}
Our first application deals with the general scoring scheme as
introduced in Section \ref{sec:pre}. Thus let $\A$ be a finite alphabet and $X$,
$Y$ be independent i.i.d. sequences so that any letter has positive
probability, i.e.
$$P(c):=P(X_1=c)>0,\quad \forall c\in \A.$$
Clearly now $\X_n=\A^n$. Let $S: \A \times \A \to \mathbb{R}^+$ be a scoring function. Let
$A<\infty$ be the maximal value of the scoring function, i.e.
$\max_{a,b}S(a,b)\leq A$. We naturally assume that the gap price
does not exceed $A$, i.e. $\delta\leq A$. Now, it is easy but
important to observe that changing one letter in the sequence $X$,
say $X_1$, decreases the score at most by $A$ units. Indeed, if
$X_1$ was not included any optimal alignment, then changing it does
not decrease the score. If an optimal alignment includes $X_1$, then
after the change, the previous alignment (which now need not to be
optimal any more) scores at most $A$ units less than before the
change. And the new optimal alignment cannot score less.
\paragraph{The random transformation.} In this setup, the random transformation is the following. Recall
that $Z$ stands for the pair of sequences $(X,Y)$.  We choose two
specific letters $a$ and $b$ from the alphabet $\A$. Given the pair
$Z$ such that at least one of the sequences contain at least one
$a$, we choose a letter $a$ from $Z$ with uniform distribution and
change it into the letter $b$. Hence $\tilde Z$ and $Z$ differ from
one letter only and as just explained, the maximum decrease of score
is at most $A$, i.e. $L(\tilde Z)-L(Z)\geq -A$ i.e. the condition
{\bf A2} is satisfied. Choosing such a transformation is motivated
by the following result (c.f.\ Theorem 5.1 and Theorem 5.2 in \cite{goetze}):
\begin{theorem}\label{app1}
Suppose there exist letters $a,b \in \A$ such that
\begin{equation}\label{mimi}
\sum_{c\in\mathbb{A}}P(c)\big(S(b,c)-S(a,c)\big)>0.
\end{equation}
Then, there exist constants $\delta_0>-\infty$, $\epsilon_o>0$,
$\alpha>0$ and $n_0<\infty$ such that
\begin{equation}\label{L}
P\big(E[{L}(\tZ)-L(Z)|Z]\geq\epsilon_o \big) \geq 1-e^{-\alpha n}
\end{equation}
given $\delta<\delta_0$ and $n\geq n_0$.\end{theorem}
{\bf Remarks:}
\begin{enumerate}
\item For the two-letter alphabet $\A=\{a,b\}$, condition
(\ref{mimi}) says
$$\big(S(b,a)-S(a,a)\big)P(a)+\big(S(b,b)-S(b,a)\big)P(b)>0.$$
When $S(b,b)=S(a,a)\ne S(a,b)=S(b,a)$, then (\ref{mimi}) is
satisfied if and only if $P(a)\ne P(b)$. For, example when
$S(b,b)=S(a,a)>S(b,a)=S(b,a)$, then (\ref{mimi}) holds if $P(a)<
P(b)$.
  \item The condition $\delta<\delta_0$ means that the gap penalty
  $-\delta$ has to be sufficiently large. Intuitively, the larger
  the gap penalty (smaller the gap price), the less gaps in optimal alignment so that the
  optimal alignment is closer to the pairwise comparison (Hamming distance). Some methods for determining a sufficient
  $\delta_0$, as well as some examples, are discussed in \cite{goetze}. We believe that the
  assumption on $\delta$ can be relaxed so that Theorem \ref{app1}
  holds under more general assumptions.
\end{enumerate}
In this section, we shall assume that there exists letters $a,b\in
\A$ so that the random transformation satisfies {\bf A1}
(equivalently, (\ref{L})). We shall show that all other assumptions
of Theorems \ref{gen1} and \ref{gen2} are fulfilled. We start with
the general case $|\A|>2$, the case $|\A|=2$ will be discussed in
the end of the present section.
\subsection{The case $|\A|>2$}\label{s:G} Let $\A=\{a,b,c_1,\ldots,c_l\}$,
where $l\geq 1$. The letters $a$ and $b$ are the ones used in the
random transformation. Let
$$q_j:={P(c_j)\over 1-P(a)-P(b)},\quad j=1,\ldots,l.$$
With $N_a$ and $N_b$ being the random number of $a$'s and $b$'s in
$X_1,\ldots,X_n,Y_1,\ldots,Y_n$, we define the random variables
$$U:=N_b,\quad V:=N_a+N_b.$$
For any $z\in \X_n\times \X_n$, thus $\mathfrak{u}(z)$ and $\mathfrak{v}(z)-\mathfrak{u}(z)$ stand
for the number of $b$'s and the number of $a$'s in both strings,
respectively. The random transformation applied on $z$ changes an
randomly chosen $a$ into a letter $b$, hence the transformation
increases $\mathfrak{u}(z)$ by one, whilst $\mathfrak{v}(z)$ remains unchanged.\\
Clearly the possible values for $U$ and $V$ are $\{0,\ldots,2n\}$
and the only restriction to $(U,V)$ is that $U\leq V$. Hence, in
this case $\S^U=\S^V=\{0,\ldots,2n\}$,
$$\S:=\{(u,v)\in \{0,\ldots,2n\}\times \{0,\ldots,2n\}: u\leq v\}$$
and for any $v$,
$$\S(v)=\{0,\ldots,v\}.$$
For any $z\in \X_n\times \X_n$,
\begin{align}\label{spr}\nonumber
P(Z=z|U=u,V=v)&=P(Z=z|N_a=v-u,N_b=u)\\ &=\left\{
                                      \begin{array}{ll}
                                        {\prod_{j=1}^l q_j^{m_j(x)}}{2n \choose{u}\;{v-u}\;{2n-v}}^{-1}, & \hbox{if $\mathfrak{u}(z)=u,\mathfrak{v}(z)=v$} \\
                                        0, & \hbox{else}
                                      \end{array}
                                    \right . ,
\end{align}
where $m_j(z)$ is the number of $c_j$-colored letters in $z$.
\paragraph{The sets ${\U}(v)$ and ${\V}$.} Note that $U\sim
B(2n,P(b))$ and $V\sim B(2n,P(a)+P(b))$. Also note that for any
$v>0$,
$$U_{(v)}\sim B(v,p_b),$$
where $p_b={P(b)\over P(a)+P(b)}.$ Let $p:=P(a)+P(b)$ and let
\begin{align*}\label{uvc}
\V& :=\left[2np-\sqrt{2n}, 2np+\sqrt{2n}\right]\cap \S^V,\\
\U(v)&:=\left[v p_b-\sqrt{v}, v p_b+\sqrt{v}\right]\cap \S(v).
\end{align*}
Now it is clear that the condition {\bf 3)} of Theorem \ref{gen1} is
fulfilled with $k_o=1$.\\\\
With the help of Chebyshev's inequality, it is straightforward to
see that for any $n$,
\begin{equation}\label{pbound}
P(V\in \V)\geq 1-p(1-p),\quad P(U_{(v)}\in \U(v))\geq
1-p_b(1-p_b).\end{equation} Clearly, there exists
 $v_o(p_b)$ so that $vp_b+\sqrt{v}<v$, whenever
$v>v_o$. Thus, there exists $n_1$ so that for every $n>n_1$ and
$v\in \V$, it holds that $v>v_o$. In particular, for every $n>n_1$
and for every pair $(u,v)$ such that $v\in \V$, $u\in \U(v)$, it
holds $v>u$.
\begin{lemma}\label{lemmaa2}
There exist an universal  constant $a>0$   and $n_2>n_1$ such that
for any $n>n_2$, for any $v\in \V$ and $u\in \U(v)$,  it holds
\begin{equation}\label{A2}
P(U=u, V=v) \ge \frac{1}{a n}.\end{equation}
\end{lemma}
\begin{proof} The proof is based on Lemma \ref{binomlclt}. Since
$U_{(v)}\sim B(v,p_b),$ by (\ref{binom}) there exists $v_o$ and a
positive constant $b_1$ so that for any $v>v_o$,
\begin{equation}\label{a1}
P(U_{(v)}=u)\geq {1\over b_1 \sqrt{v}},\quad \forall u\in \U(v).
\end{equation}
Secondly, since $V\sim B(2n,p)$, there exists $n_{2,1}$ so that for
every $n>n_{2,1}$,
\begin{equation}\label{a2}
P(V=v)\geq {1\over b_2 \sqrt{2n}},\quad \forall v\in \V.
\end{equation}
Take now $n_2>n_{2,1}$ so big that for any $n>n_2$ it holds
$2np-\sqrt{2n}>v_o$. Then, for every $v\in \V$, (\ref{a1}) and
(\ref{a2}) both hold. Thus, for any $n>n_2$, $v\in \V$ and $u\in
\U(v)$, we have
$$P(U=u, V=v)=P(U_{(v)}=u)P(V=v)\geq {1\over b_1 \sqrt{v} b_2 \sqrt{2n}}\geq {1\over (b_1b_2)\sqrt{2n(2np+\sqrt{2n})}}\geq {1\over a n},$$
where the constant $a$ can be taken as  $2b_1b_2\sqrt{p+1}$.
\end{proof}
\begin{lemma}\label{condvar2}
There exists a finite $n_3$  and a constant $d>0$ such that
$n_3>n_2$ and for every $n > n_3$ and $v\in \V$, it holds
\begin{equation}\label{kass}
\v[U_{(v)}|U_{(v)}\in \U(v)]\geq d n .
\end{equation}
\end{lemma}
\begin{proof}
From Lemma \ref{binomlclt}, we know that  there exists $c_1$ and
$v_o$, so that
\begin{equation}\label{c1}
\v[U_{(v)}|U_{(v)}\in \U(v)]\geq c_1  v,
\end{equation}
provided $v>v_o$. Let $n_{3,1}$ be such that for every $n>n_{3,1}$
$2np-\sqrt{2n}>   v_o.$ Then,  for any $n>n_{3,1}$ an any $v\in
\V$, we have that $v>v_o$ so that (\ref{c1}) holds and
\begin{equation}\label{c2}
\v[U_{(v)}|U_{(v)}\in \U(v)]\geq c_1(2pn-\sqrt{2n}).
\end{equation}
Finally take $n_3>n_{3,1}$ so big that for a constant $d>0$,
$c_1(2pn-\sqrt{2n})\geq d n$, provided $n>n_3$.\end{proof}
Finally we prove {\bf A3} for that particular model.
\begin{lemma}\label{tilde2} Let  $(u,v)\in \S$ be such that $v>u$. Let $Z\sim
P_{(u,v)}$ Then $\tZ\sim P_{(u+1,v)}$
\end{lemma}
\begin{proof}
For any $z\in \X_n\times \X_n$, let the set ${\cal A}(z)$ consists
of possible outcomes after applying the random transformation to
$z$. Since the transformation changes an $a$ into $b$, the number
of different outcomes equals to the number of $a$'s in $z$, thus
$|{\cal A}(z)|=\mathfrak{v}(z)-\mathfrak{u}(z)$. Since the transformation picks the
letters uniformly, we obtain that for any $\tilde{z}\in {\cal
A}(z)$,
\begin{equation}\label{uz}
P(\tZ=\tilde{z}|Z=z)={1\over \mathfrak{v}(z)-\mathfrak{u}(z)}.
\end{equation}
Let the set ${\cal B}(\tilde{z})$ consist of all these pairs of
strings that could result $\tilde{z}$ after the transformation:
$\mathcal{B}(\tilde{x}):=\{ x \in \mathcal{X} : \tilde{x} \in
\mathcal{A}(x)\}$. Since before transformation one of $b$'s in
$\tilde{z}$ was $a$, clearly $|{\cal B}(\tilde{z})|=\mathfrak{u}(\tilde{z})$.
Recall that $U$ and $V$ are the functions of $Z$. Let $Z\sim
P_{(u,v)}$. We aim to find
\begin{align}\label{summa}
P(\tZ=\tilde{z})=P(\tZ=\tilde{z}|U=u,V=v)&=\sum_{z\in {\cal
B}(\tilde{z})}P(\tZ=\tilde{z}|Z=z)P(Z=z|U=u,V=v).\end{align} Let
$$S(u,v):=\{z: \mathfrak{u}(z)=u,\mathfrak{v}(z)=v\}.$$
Clearly the right hand side of (\ref{summa}) is zero, if
$$S(u,v)\cap {\cal B}(\tilde{z})=\emptyset.$$
This simply  means that the string $\tilde{z}$ does not satisfy at
least one of the following equalities:
$$\mathfrak{u}(\tilde{z})=u+1,\quad \mathfrak{v}(\tilde{z})=v.$$
Let us now assume that $\tilde{z}$ satisfies both equalities above.
In particular, $|{\cal B}(\tilde{z})|=u+1$ and any element in ${\cal
B}(\tilde{z})$ is such that the number of $b$'s is  $u$ and the
number of $a$'s is $v-u$ and the number of all $c_j$ equal to that
of $\tilde{z}$. i.e. $m_j(z)=m_j(\tilde{z})$ $\forall \tilde{z}\in
{\cal B}(\tilde{z})$. Clearly ${\cal B}(\tilde{z})\subset S(u,v)$.
By (\ref{spr}), thus
\begin{align*}
P(\tZ=\tilde{z}|U=u,V=v)&=\sum_{z\in {\cal
B}(\tilde{z})}P(\tZ=\tilde{z}|Z=z)P(Z=z|U=u,V=v)\\
&={1\over v-u}|{\cal B}(\tilde{z})| {2n \choose {v-u}\,\, {u}\,\, {2n-v} }^{-1}   {\prod_{j=1}^l q_j^{m_j(\tilde{z})}}\\
&={u+1\over v-u}{2n \choose {v-u}\,\, {u}\,\, {2n-v} }^{-1}\prod_{j=1}^l q_j^{m_j(\tilde{z})} \\
&={u!(u+1)(v-u)!(2n-v)!\over (v-u)2n!}\prod_{j=1}^l q_j^{m_j(\tilde{z})}\\
&={(u+1)!(v-u-1)!(2n-v)!\over 2n!}\prod_{j=1}^l q_j^{m_j(\tilde{z})}\\
&=\prod_{j=1}^l q_j^{m_j(\tilde{z})}{2n \choose {v-u-1}\,\,
{u+1}\,\, {2n-v} }^{-1}.\end{align*} By (\ref{spr}), $\tZ\sim
P_{(u+1,v)}$.
\end{proof}
\begin{theorem}\label{main2} Assume that the random transformation satisfies ${\bf
A 1}$. Then there exists an universal constant $b>0$ and
$n_6<\infty$ so that for every $n>n_6$, it holds
\begin{equation}\label{varbound}
\v[L(Z)]\geq b \cdot n.
\end{equation}
\end{theorem}
\begin{proof} Let us first check the assumptions of Theorem
\ref{gen2}. {\bf A1} holds by the assumption. As explained above,
the random transformation is such that ${\bf A2}$ holds. Let now
$n_2$ be as in Lemma \ref{lemmaa2} and $n_3$ as in Lemma
\ref{condvar2}. Recall that $n_1<n_2<n_3<\infty$. Hence, for any
$n>n_3$, (\ref{A2}) and (\ref{kass}) hold; moreover, from
(\ref{pbound}), it follows that with $d_1=1-p(1-p)$ and
$d_2=1-p_b(1-p_b)$, the inequalities (\ref{dd1}) hold and for any
pair $(u,v)$ where $v\in \V$ and $u\in \U(v)$, we have that $v>u$.
The latter ensures that the random transformation is possible and
Lemma \ref{tilde2} now establishes {\bf A3}. Therefore, for every
$n>n_3$, the assumptions of Theorem \ref{gen2} are fulfilled and so
there exists $n_5>n_3$ so that for for every $n>n_5$, the
assumptions of Theorem \ref{gen1} hold with $\delta={\e_0\over
2}.$\\
We now apply Theorem \ref{gen1}. As just showed, the assumption {\bf
1}) holds for any $n>n_5$; as explained in Subsection \ref{s:G}, the
assumption {\bf 3)} holds with $k_o=1$. By (\ref{kass}), $\psi_n=d
n$. By Theorem \ref{gen1}, thus, the lower bound (\ref{varbound})
exists with
$$b={ \epsilon_o^2 d_1 d_2 d\over 4}.$$
\end{proof}
\subsection{The case $A=\{a,b\}$}
This case is easier.  The only random variable is $U$, formally we
can take $V\equiv 2n$. Then (\ref{spr}) is
$$P(Z=z|U=u)= P(Z=z|U=u,V=2n) = {\binom{2n}{u}}^{-1}I_{\{\mathfrak{u}(z)=u\}}.$$
Now take $\V=\{2n\}$ and
$$\U=\U(2n)=[2np_b-\sqrt{2n},2np_b-\sqrt{2n}]\cap S_n^U.$$  Then everything holds true: there clearly exists $n_1$
so that $u<v=2n$, whenever $n>n_1$ and $u\in \U$; the bound
(\ref{A2})  holds with $a=b_1$ (and, in fact $n^{-{1\over 2}}$
instead of $n^{-1}$); the proof of Lemma \ref{condvar2} is simply
(\ref{c1}) and the proof of Lemma \ref{tilde2} holds true with
$\mathfrak{v}(z)=2n$. Thus Theorem \ref{main2} holds.
\section{The length of the LCS of random i.i.d.\ block sequences}\label{sec:LCS}
In this section we are interested in the fluctuations of $L(Z)$ for the score function as in \eqref{LCS-scoring}, where $Z=(X,Y)$ are binary sequences having a certain random block structure. This random block structure was first considered in \cite{FT1}. In the present article, we consider a random block structure which is a generalization of the model in \cite{FT1}. We are able to show that the length of the longest common subsequence of two sequences having this random block structure grows linearly by following the general two-step approach. Therefore, we confirm in this setting Waterman's conjecture.
\\
\\
\noindent
Ofte in this section $x,y$ stand for binary strings of length $n>0$. We start by getting a bit more familiar with the LCS of $x$ and $y$. First, note that the LCS of $x$ and $y$ and the alignment generating it are not necessarily unique, but its length is unique.
\begin{example}
Let $x=100101100001101$ and $y=111000010101110$. The length of the LCS of $x$ and $y$ is $L_{15}(x,y)=10$. A candidate for the LCS of $x$ and $y$ is the string $1000100111$. This LCS could have came (but not exclusively) from any of the following alignments:
\[ \begin{array}{c|c|c|c|c|c|c|c|c|c|c|c|c|c|c|c|c|c|c|c|c|c|c}
x&&1&-&-&0&0&1&0&1&-&1&0&0&0&-&0&1&1&-&0&1&-\\ \hline
y&&1&1&1&0&0&-&0&-&0&1&-&0&-&1&0&1&1&1&-&1&0\\ \hline
{\bf LCS}&&{\bf 1}& & &{\bf 0}&{\bf 0}&&{\bf 0}&&&{\bf 1}&&{\bf 0}&&&{\bf 0}&{\bf 1}&{\bf 1}&&&{\bf 1}&
\end{array}\]
\[ \begin{array}{c|c|c|c|c|c|c|c|c|c|c|c|c|c|c|c|c|c|c|c|c|c|c}
x&&1&-&-&0&0&1&0&1&-&1&0&0&0&-&0&1&-&1&0&1&-\\ \hline
y&&1&1&1&0&0&-&0&-&0&1&-&-&0&1&0&1&1&1&-&1&0\\ \hline
{\bf LCS}&&{\bf 1}& & &{\bf 0}&{\bf 0}&&{\bf 0}&&&{\bf 1}&&&{\bf 0}&&{\bf 0}&{\bf 1}&&{\bf 1}&&{\bf 1}&
\end{array}\]
Another candidate for an LCS of $x$ and $y$ is the string $1000000110$.
\end{example}

\vspace{12pt}
\noindent
We now introduce the random block model:

\subsection{The 3-multinomial block model}\label{sec:block-model}
We say that a block of zeros of length $m \in \mathbb{Z}_+$ in $x$ is a run of $0$'s of maximal length between two ones, except for the block of zeros at the beginning of $x$, which only has a 1 inmediatly to its left. We consider the analog convention for a block of ones of length $m$ in $x$, as well as for any binary sequence. Let $l \ge 2$ and $q_1,q_2,q_3 \in (0,1)$ be parameters such that $q_1+q_2+q_3=1$. Let $(W_k)_k$ and ${(W'_k)}_k$ be two i.i.d.\ sequences of random variables taking values on $\{l-1,l,l+1\}$, independent of each other, with distribution 
\[ P(W_k=l_i)=P(W_k'=l_i) = q_i\,,\qquad\forall\,k \ge 1, \,\,i\in \{1,2,3\}\]
where $l_1:=l-1,l_2:=l$ and $l_3:=l+1$. Let $(w_k)_k$ be a realization of $(W_k)_k$. Let us construct $x^\infty=x_1x_2x_3\cdots$ an infinite binary sequence depending on $(w_k)_k$ as follows: We choose $\vartheta \in \{0,1\}$ with probability $1/2$, independently from everything else. Then, we built up the first block in $x^\infty$ as a block of $\vartheta$'s with lenght $w_1$, the second block in $x^\infty$ as a block of $(1-\vartheta)$'s with length $w_2$, the third block in $x^\infty$ as a block of $\vartheta$'s with length $w_3$, and so on. We built up, in a completely analog way, the sequence $y^\infty$ based on a realization ${(w'_k)}_k$ of ${(W'_k)}_k$. After this, for a given $n>0$, let us define $\hat{x}((w_k)_k):=x^\infty[1,n]:=x_1\cdots x_n$ and $\hat{y}((w'_k)_k):=y^\infty[1,n]:=y_1\cdots y_n$, namely the first $n$-bytes of the infinite sequence $x^\infty$, respectively $y^\infty$. Note that the last run of the same symbol in $\hat{x}((w_k)_k)$ (resp.\ in $\hat{y}((w'_k)_k)$) is not a block according to our definition, or though its length $r $ is such that $r\in \{1,\dots,l+1\}$. Naturally, $\hat{x}((w_k)_k)$ (or equivalently $\hat{y}((w'_k)_k)$) induces the set $\mathcal{X}_n \subset \{0,1\}^n$ of binary sequences having blocks only with length either $l-1, l$ or $l+1$ and a last run of the same symbol with length $r$ such that $r\in\{1,\dots,l+1\}$. Let us denote by $X:=\hat{x}((W_k)_k):=X_1\cdots X_n$ (resp. by $Y:=\hat{y}({(W'_k)}_k):=Y_1\cdots Y_n$) the associated random binary sequence of length $n$ whose realization is an element of $\mathcal{X}_n$. The process of allocating the blocks can be seen as independently drawing balls of 3 different colours from an urn, where a ball of colour $i$ has probability $q_i$ to be picked up, $i=1,2,3$. That is why we call this the $3$-multinomial block model. For $k \in \{l-1,l,l+1\}$ and $x \in \mathcal{X}_n$, let $b_k(x)$ be the number of blocks of length $k$ in $x$,
and denote $B_k:=b_k(X)$ the associated random variable.

\begin{example}
Take $l=3$ and let $W_1=2, W_2=3,W_3=2, W_4=4,W_5=3,\dots$ Suppose that we get $\vartheta=0$, so then $x^\infty=00111001111000\dots$. For $n=13$, we get the sequence $x= 0011100111100$ such that $b_2(x)=2, b_3(x)=1$ and $b_4(x)=1$, with a rest at the end of length $r=2$.
\end{example}

\noindent
Let us define the following three new random variables:
\begin{eqnarray}
T&:=&B_l+B_{l-1}+B_{l+1}\label{T}\\
U&:=&B_l-B_{l-1}-B_{l+1}\label{U} \\
R&:=&n-(\,l\, B_l+(l+1)\, B_{l+1}+(l-1)\, B_{l-1}\,).\label{R}
\end{eqnarray}
Given $(b_{l-1}(x),b_l(x),b_{l+1}(x))$, let us denote by
$(t(x),u(x),r(x))$ the solution of the linear system
\begin{eqnarray*}
t(x)&=&b_l(x)+b_{l-1}(x)+b_{l+1}(x)\\
u(x)&=&b_l(x)-b_{l-1}(x)-b_{l+1}(x)\\
r(x)&=&n-(\,l\, b_l(x)+(l+1)\, b_{l+1}(x)+(l-1)\, b_{l-1}(x)\,).
\end{eqnarray*}
The other way around, given any realization $(t,u,r)\in
\mathbb{Z}^3$ of $(T,U,R)$, let us denote by
$$(b_{l-1}(t,u,r),b_l(t,u,r),b_{l+1}(t,u,r))$$ the solution of the
linear system:
\begin{equation}\label{linearTZ}
\left(\begin{array}{c} b_{l-1}(t,u,r) \\ b_l(t,u,r) \\ b_{l+1}(t,u,r) \end{array}\right)=
\left(\begin{array}{cc} (2l+1)/4 & -1/4 \\ 1/2 & 1/2 \\ -(2l-1)/4 & -1/4\end{array}\right) \left( \begin{array}{c} t
 \\ u \end{array}\right)+ \left(\begin{array}{c} -(n-r)/2 \\ 0 \\ (n-r)/2 \end{array} \right).
\end{equation}
This means that we have a one-to-one correspondence between the
random variables $(B_{l-1},B_l,B_{l+1})$ and $(T,U,R)$, which will
be often used in what follows.
\paragraph{The 3-multinomial distribution.} We can compute the distribution of $X$ by taking into account its block structure. In order to do so, let us define the function
\[p(r):=P(W\geq r),\qquad\mbox{for $r\in \{1,\dots,l+1\}$.}\]
\noindent
Clearly, $p(r)=1$ when $r \in \{1,\ldots,l-1\}$, but
$p(l)=q_2+q_3$ and $p(l+1)=q_3$. Now, we see that for any $x\in \X_n$ it holds
\begin{equation}\label{probX}
P(X = x) = {1\over 2} q_1^{b_1(x)}q_2^{b_2(x)}q_3^{b_3(x)}p(r(x)),
\end{equation}
where the factor ${1\over 2}$ is needed because to every fixed
block-sequence corresponds two sequences in $\X_n$, both having the
same probability (it is the choosing of the colour of the first block with probability $1/2$). Moreover, e.g.\ by the urn analogy, we find that the joint distribution of $(B_{l-1},B_l,B_{l+1})$ can be computed as following: given $b_1,b_2,b_3\in \mathbb{N}$ it holds
\begin{equation}\label{multiB}
P(B_{l-1}=b_1,B_l=b_2,B_{l+1}=b_3)= \sum_{x \in S(b_1,b_2,b_3)} P(X=x) = {b_1+b_2+b_3 \choose b_1\;\;b_2\;\;b_3}q_1^{b_1}q_2^{b_2}q_3^{b_3}p(r),
\end{equation}
where
\[S(b_1,b_2,b_3):=\{ x\in\mathcal{X}_n : b_{l-1}(x)=b_1,b_l(x)=b_2,b_{l+1}(x)=b_3\} \]
and $r=n-(l-1)b_1-l b_2-(l+1)b_3.$ Note that the factor ${1\over 2}$
disappears. So, combining \eqref{probX} and \eqref{multiB} we
naturally get that for $x \in \mathcal{X}$ it holds
\begin{equation}\label{probXcond}
P(X=x | B_{l-1}=b_1,B_l=b_2,B_{l+1}=b_3) = {1\over 2}\,{b_1+b_2+b_3 \choose b_1\;\;b_2\;\;b_3}^{-1}1_{S(b_1,b_2,b_3)}(x).
\end{equation}
Note also that, from \eqref{linearTZ}, we can even compute the joint
distribution of $(T,U,R)$ as follows:
\begin{equation}\label{TZRdist}
P(U=u,T=t,R=r) ={ t \choose b_{l-1}(t,u,r)\;b_l(t,u,r)\;b_{l+1}(t,u,r)}q_1^{b_{l-1}(t,u,r)}q_2^{b_l(t,u,r)}q_3^{b_{l+1}(t,u,r)}p(r).
\end{equation}
\subsection{Fluctuations of the length of the LCS in the 3-multinomial block model}\label{subsec:u}
Let $Z=(X,Y)$ be a vector of binary sequences, where each component has the previously introduced random block structure. Let us identify $U$ defined in \eqref{U} with the random variable $\mathfrak{u}(Z)$ as well as the vector $(T,R)$ with the random variable $\mathfrak{v}(Z)$. Therefore, in what follows $(U,V)=(T,U,R)$ and its support $\S$ consists of triples $(t,u,r)$ so that $P(U=u,T=t,R=r)>0$. We would like to use Theorem \ref{gen1}, so we must look for sets $\U$ and $\V$ such that the conditions {\bf 1), 2)} and {\bf 3)} are satisfied. 
For any $c>0$, define
\begin{align*}\label{uvc}
\U^c&:=\left[ \frac{n}{\mu}(q_2-q_1-q_3)-c\sqrt{n}, \frac{n}{\mu}(q_2-q_1-q_3)+c\sqrt{n}\right]\cap \mathbb{Z},\\
 {\mathcal{T}}_n^c&:= \left[ \frac{n}{\mu}-c\sqrt{n},\frac{n}{\mu}+c\sqrt{n}\right]\cap \mathbb{Z},\\
\V^c& :=\big(\mathcal{T}_n^c\times \{1,\ldots, l+1\}\big)\cap\S^V.
\end{align*}

\vspace{12pt}
\noindent
Note that the notation $\U^c$ (resp.\ $\mathcal{T}_n^c$ and $\V^c$) means that the set $\U$ explicitely depends on the constant $c>0$, and has nothing to do with the notation for the complement of a set. Recall the right hand side of \eqref{claimU} and the 2.\ Remark after Theorem \ref{gen1}. A first observation is that, uniformly on $n$, $P(U \in \U^c,V \in \V^c)$ is bounded by below by a constant:
\begin{lemma}\label{lemma09}
There exist universal constant $c>0$  (not depending on $n$)  and
$n_0 < \infty$ such that for every $n > n_0$ it holds
\[P\big(U\in \U^c, V\in \V^c\big) \ge 0.9  .\]
\end{lemma}

\vspace{6pt}
\noindent
The proof is a rather straightforward application of large deviation
techniques, and therefore it is contained in the Appendix. We shall also need the following lemma (proven in the Appendix):
\begin{lemma}\label{RMpossible}  There exist an universal constant $\alpha>0$
and $n_1 \in (n_0,\infty)$ such that for every $n>n_1$ and $(u,v) \in
\mathcal{S}_n \cap (\mathcal{U}_n\times \mathcal{V}_n)$, it holds
\begin{eqnarray}\label{ulboundsbk}
\left| b_{l_i}(u,v)- q_i\frac{n}{\mu} \right| \le \alpha \sqrt{n}
\quad\mbox{and}\quad b_{l_i}(u,v) \ge 1, \quad\forall\,i=1,2,3.
\end{eqnarray}
\end{lemma}

\vspace{12pt}
\noindent 
In what follows, we shall take $\U:=\U^c$ and  $\V=\V^c$, where
$c>0$ is as in Lemma \ref{lemma09} and we shall take $n>n_0$. Recall the definition $\U(v):=\U \cap \S(v)$, for $v \in \V$. We show now how the conditions of Theorem \ref{gen1} are fulfilled:
\paragraph{Condition 3).}  Lets us start by showing that $u\in \U(v)$ implies $u+i\notin
\U(v)$ for $i=1,2,3$. Indeed given $u \in \U(v)$, it is enough to realize that $b_{l-1}(u+i,v)$ is not an integer if $i=1,2,3$ but $b_{l-1}(u+4,v)$ is an integer, where $b_{l-1}(u,v),b_{l}(u,v),b_{l+1}(u,v)$ are the integer solutions of the system \eqref{linearTZ}. Next, from Lemma \ref{RMpossible}, it follows (among other things) that, if $n$ is big enough, $u+4\in \S(v)$ for every $v\in \V$ and $u\in \U(v)$, which finally implies that $\U(v)$ is of the form (\ref{lattice}) with $k_o=4$ so that the condition ${\bf 3)}$ of Theorem \ref{gen1} holds with
$k_o=4$. Let us explain all the last argument. For every $n>n_1,$ $v\in \V$ and $u\in \U(v)$, we have that necessarily $u+4\in S_n(v)$, so $(u,v) \in \mathcal{S}_n$, therefore there exists at least one possible outcome $x\in \X_n$ so that $\mathfrak{u}(x)=u$ and $\mathfrak{v}(x)=v$. Moreover, from (\ref{ulboundsbk}), it follows that $b_{l_i}(u,v)\geq 1$ for every $i=1,2,3$. Thus, in $x$ there are at least one block from every size $\{l-1,l,l+1\}$. Deleting from $x$ one bit from a block of the length $l+1$ and adding one bit to the block of length $l-1$, we turn both them to the blocks of the length $l$. This transformation does not change the number of blocks and the sum of the lengths of all blocks, so we have another possible outcome of $X$, say $\tilde{x}$ with $\mathfrak{u}(\tilde{x})=u+4$ and $\mathfrak{v}(\tilde{x})=v$. Hence $(u+4,v) \in \mathcal{S}_n$ as well. If we keep on doing so, there exist an integer $m(v,n)>1$ and $u_o(v,n)\in \mathbb{Z}$ such that 
\begin{equation}\label{um}
\U(v)=\{u_o(v,n)+4i:\,\, i=1,\ldots, m(v,n)\}.
\end{equation}
\noindent It is not hard to see that, for every $n>n_1$ and $v\in \V $, the integer $m(v,n)$ satisfies
\begin{equation}\label{m}
{2c \sqrt{n}\over 4}-2 < m(v,n) < {2c \sqrt{n}\over 4}+ 2.\end{equation}

\paragraph{Condition 2).} 
Recall that $(u,v)=(t,u,r)$. Let $U_{(v)}$ denote a random variable
distributed as $U$ given $V=v$, namely for every $z\in \mathbb{Z}$ it holds
\[P(U_{(v)}=z)=P(U=z|V=v).\]
For every $n>n_1$ and $v\in \V $, let us define:
$$p_n(i):=P\big(U_{(v)}=u_o(v,n)+4i\,|\,U_{(v)}\in\U \big),\quad i=1,\ldots, m(v,n).$$
The following lemma shows that the ratio $ p_n(i+1) / p_n(i)$ tends
to one with speed $O(n^{-{1\over2}})$. The proof, given in the Appendix,
is heavily based on the following well-known inequalities:
\begin{eqnarray}\label{log}
-\frac{3x}{2} &\le& \ln (1-x), \quad\mbox{for $0<x \le 0.5$} \nonumber \\
\ln(1+x) &\le& x, \quad \mbox{for $x>-1$.}
\end{eqnarray}
\begin{lemma}\label{quotientTZR} There exists an universal constant $K<\infty
$ and $n_2>n_1$ such that for every $n>n_2$ and $v \in\V$ it holds,
\begin{equation}\label{BinScope}
1-\frac{K}{\sqrt{n}} \le \frac{p_n(i+1)}{p_n(i)} \le 1+\frac{K}{\sqrt{n}},\quad i=1,\ldots, m(v,n).
\end{equation}
\end{lemma}
Recall \eqref{phi}, 2.\ and 4.\ Remark after Theorem \ref{gen1}. We are now ready to prove that the (conditional) variance of $U$ increases linearly, i.e.\ condition {\bf 2)}:
\begin{lemma}\label{variance}
 There exist an universal constant $d>0$ and  $n_3>n_2$ so that for every $n>n_3$ and for every  $v\in
\V$ it holds,
\begin{equation}\label{varlin}
\v[U|U\in \U,V=v]=\v[U_{(v)}|U_{(v)}\in \U]=\v[U_{(v)}|U_{(v)} \in\U(v)]\geq dn.
\end{equation}
\end{lemma}
\begin{proof} Let $n>n_2$ and $v\in \V$. Recall (\ref{um}) and
(\ref{m}). There  exist constants $0<d_1<d_2$ so that for every $n$
big enough, say $n>n_3>n_2$, it holds $$ d_2 \sqrt{n}<{2c
\sqrt{n}\over 4} -6 <{2c \sqrt{n}\over 4}+ 2<    d_2 \sqrt{n}.$$
From (\ref{m}), it follows that for every $n>n_3$ and $v\in V$
\begin{equation}\label{dd}
d_1 \sqrt{n} <  m-4 < m <    d_2 \sqrt{n}.
\end{equation}
Take $n>n_3$. From  \eqref{BinScope}, it follows that
\[1-\frac{K}{\sqrt{n}} \le {p_n(i+1)\over p_n(i)} \le 1+\frac{K}{\sqrt{n}}\]
for  $i=1,\dots m-1$. Hence, for every $i,k \in \mathbb{N}$ such
that $i+k\leq m$, it holds
$$\Big(1-{K\over \sqrt{n}}\Big)^k \leq {p_n(i+k)\over p_n(i)}={p_n(i+1)\over
p_n(i)}\cdot {p_n(i+2)\over p_n(i+1)}\cdots {p_n(i+k)\over
p_n(i+k-1)} \leq \left(1+\frac{K}{\sqrt{n}}\right)^k.$$ Recall
\eqref{log}, so that from the last inequality we get
\begin{equation}\label{expEQUAL}
\exp\left(-\frac{3K}{2\sqrt{n}}k\right) \le {p_n(i+k)\over p_n(i)}\le
\exp\left(\frac{K}{\sqrt{n}}k\right).
\end{equation}
Thus, for every $1\leq i,j \leq m$, we have
\begin{equation}\label{E}
{p_n(i)\over p_n(j)}\leq \exp \left(\frac{3K}{2\sqrt{n}}|i-j|\right)\leq
\exp \left(\frac{3K}{2\sqrt{n}}m \right)\leq
\exp \left(\frac{3K}{2}d_2\right)=:E.\end{equation} From (\ref{E}), it follows
that $p_n(i)\leq (\min_i p_n(i))E$, so that
$$1=\sum_{i=1}^m p_n(i)\leq  m (\min_i p_n(i)) E$$
and we obtain that for every $i=1,\ldots,m$, it holds
$$p_n(i)\geq (\min_i p_n(i)) \geq {1\over m E}\geq {1\over Ed_2
\sqrt{n}}.$$ Now, the variance can be estimated as follows. Let
$\bar{u}:=E[U_{(v)}|U_{(v)}\in \U]$. Without loss of generality, let
us assume that $\bar{u}(v,n)\leq u_o+4({m+1\over 2})=u_o+2m+2$. Then
for every $i\geq {3m\over 4}$, it holds that
$|u_o+4i-\bar{u}|=u_o+4i-\bar{u}\geq m-2$. Then
\begin{align*}
\v [U_{(v)}|U_{(v)}\in \U]&=\sum_{i=1}^m(u_o+4i-\bar{u})^2p_n(i)\geq
\sum_{i\geq {3 m\over 4}}^m(u_o+4i-\bar{u})^2p_n(i)\\ &\geq ({m\over
4}-1)(m-2)^2 {1\over Ed_2 \sqrt{n}}\geq {d^3_1\over 4 E
d_2}n.\end{align*}
\end{proof}
\paragraph{Condition 1).} The strategy is to look for a random mapping which satisfies assumptions {\bf A1, A2} and {\bf A3}, so that we check condition {\bf 1)} by applying Theorem \ref{gen2}. Recall that to apply Theorem \ref{gen2},  we additionally need that every point in the set $S_n\cap (\U\times \V)$ has to have sufficiently big probability so that the condition (\ref{LCLT}) is fulfilled. The following lemma, also  proven in the Appendix, shows that the defined sets $\U$ and $\V$ indeed satisfy this additional condition:
\begin{lemma}\label{lemmaa}
There exist an universal  constant $a>0$   and $n_4>n_3$ such that
for any $n>n_4$ and $(u,v) \in \mathcal{S}_n \cap (\U\times \V)$ it
holds
\[ P(U=u, V=v) \ge \frac{1}{a n}.\]
\end{lemma}
Let us finally introduce the random transformation $\mathcal{R}$: Take $z=(x,y) \in\mathcal{X}_n\times\mathcal{X}_n$. Then in $x$, we choose uniformly at random a block of length $l-1$ (among all the $b_{l-1}(x) \ge 1$ available blocks of length $l-1$) and turn it to a block of length $l$. At the same time and independent from our previous choice, we choose uniformly at random a block of length $l+1$ (among all the $b_{l+1}(x) \ge 1$ available blocks of length $l+1$) and also turn it to a block of length $l$. We do not perform any change in $y$. Following our initial convention, $\tilde{z}:=\mathcal{R}(z)=(\tilde{x},y) \in \mathcal{X}_n \times\mathcal{X}_n$ is the sequence after applying this transformation.
\begin{example}
As in a previous example with $l=3$ and $n=13$, let us take $x= 0011100111100$ such that $b_2(x)=2, b_3(x)=1$ and $b_4(x)=1$, with a rest at the end of length $r=2$. In $x$, there are only two blocks of length $l-1=2$ to pick from, each with probability $1/2$, namely $00$ (most left one) or $00$ (following most left one), and only one block of length $l+1=4$ to pick from, with probability $1$, namely $1111$. Let us suppose that $\mathcal{R}$ picks up $00$ (most left one) and $1111$, then $\tilde{x}$ will look like this $\tilde{x}=0001110011100$.
\end{example}
Note naturally that $b_l(\tilde{x})=b_l(x)+2$, $b_{l-1}(\tilde{x})=b_{l-1}(x)-1$ and $b_{l+1}(\tilde{x})=b_{l+1}(x)-1$, so that for $k_o=4$ the condition \eqref{shift} is satisfied.

\vspace{12pt}
\noindent
We will prove that $\mathcal{R}$ satisfies assumptions {\bf A3} and {\bf A2}, but we do not prove in this paper that $\mathcal{R}$ satisfies assumption {\bf A1}, because it would be too long (see 2.\ Remark after \eqref{impli}) and the proof deals with another issues of random sequences comparison, which are different from the fluctuations ones we try to be focused on along the present article. It means that our main result, i.e.\ Theorem \ref{LCSthm}, delivers the linear fluctuations result assuming that {\bf A1} is fulfilled. This is not restrictive, since in \cite{FT1} assumption {\bf A1} was already proven for the special case $q_1=q_2=q_3=1/3$ (as well as the linear fluctuations result). For the sake of completness, we include here the before mentioned result with our current notation:
\begin{lemma}\label{phdfirststep}
Let $q_1=q_2=q_3=1/3$. There exist $n_0 < \infty$ and a constant $\alpha \in (0,1)$ not depending on $n$ but on $l$, such that for every $n>n_0$ the event
\begin{equation*} \label{enbeta}
E[L(\tZ)-L(Z)|Z]\geq \e
\end{equation*}
happens with probability bigger than $1-\exp[-n^{\alpha}]$.
\end{lemma}

\begin{remark}[for readers who want to dig in the details of \cite{FT1}]
As we have mentioned, the LCS of two sequences might have associated many different alignments, which we call optimal alignments. Lemma 4.2.1 in \cite{FT1} showed that the set of realizations of $X$ and $Y$ such that their optimal alignments leave out at most a proportion $q_0$ of blocks, where $q_0 > \frac{4}{9(l-1)}$, have probability exponentially close to one as $n \to \infty$. The entire chapter 4 in \cite{FT1} is dedicated to prove results of this type for a set of realizations of $X$ and $Y$ and their optimal alignments. Then, in Lemma 4.7.1 in \cite{FT1} it is shown Lemma \ref{phdfirststep} as in the following form: the set of realizations of $X$ and $Y$ such that their optimal alignments satisfy $E[L(\tZ)-L(Z)|Z]\geq \e_1$ has probability exponentially close to one as $n \to \infty$, for an arbitrary $\e_1>0$. It is important to note that in \cite{FT1}, all the extra conditions of Lemma 4.7.1 and the extra work through the chapter 4 and chapter 6 are devoted to relate the value of $\e_1$ with the smallest possible length of the blocks $l>0$ in order to get a sharp result for the order of the fluctuations of $L(Z)$. This relation is obtained in terms of an optimization problem which can be explicitely solved for this particular 3-multinomial model where $q_1=q_2=q_3=1/3$.
\\
The authors are working on a separate article about how to generalize Lemma \ref{phdfirststep} to the case $q_1,q_2,q_3 \in (0,1)$ such that $q_1+q_2+q_3=1$ (namely, the present and more general 3-multinomial block model).
\end{remark}

\paragraph{Assumption A3.} Recall that {\bf A3} presupposes that for any $(u,v)\in
\S\cap (\U\times \V)$, $b_{l-1}(u,v)\geq 1$ and $b_{l+1}(u,v)\geq
1$, which follows from \eqref{ulboundsbk}. The following lemma proves {\bf A3} :
\begin{lemma}\label{tilde}
Let   $(u,v)\in \S$ be such that $b_{l-1}(u,v)\geq 1$ and
$b_{l+1}(u,v)\geq 1$. If $Z \sim P_{(u,v)}$, then $\tilde{Z} \sim
P_{(u+4,v)}$.
\end{lemma}
\begin{proof} The random variables $U$, $T$ and $R$ are independent
of $Y$. Hence, $P_{(u,v)}=P^x_{(u,v)}\times P^y$, where
$P^x_{(u,v)}$ is the conditional distribution of $X$ given
$\{(U,V)=(u,v)\}$, $P^y$ is the law of $Y$ (actually $P^x=P^y$) and
$\times$ stands for the product measure. Also the block
transformation applies to $X$, only. Thus, it suffices to show that
if $X\sim  P^x_{(u,v)}$, then $\tilde{X} \sim P^x_{(u+4,v)}$. For
proving this, we follow the approach in Lemma \ref{tilde2}.\\
The first step of the proof is to explicitly compute an expression
for $P^x_{(u,v)}=P^x_{(u,t,r)}$. By \eqref{probXcond}, we have that
for any $x \in S(b_{l-1}(u,v),b_l(u,v),b_{l+1}(u,v))$ it holds:
\begin{eqnarray}
P^x_{(u,v)}(x) &=& P(X=x | U=u, T=t, R=r) \nonumber\\
&=& P(X=x | B_{l-1}=b_{l-1}(u,v),B_l=b_l(u,v),B_{l+1}=b_{l+1}(u,v)) \nonumber\\
&=&\!\!\!\! {1\over 2}{t \choose b_{l-1}(u,v)\;b_l(u,v)\;b_{l+1}(u,v)}^{-1}. \label{distXtzr}
\end{eqnarray}
The second step of the proof is to actually compute the distribution
of $\tilde{X}$. For that, we need to investigate the effect of the
block-transformation on the distribution of $X$. Let us fix $x \in
\mathcal{X}$ and denote $b_1:=b_{l-1}(x), b_2:=b_l(x)$ and
$b_3:=b_{l+1}(x)$, and $(u,v)$ its corresponding triple. Let us
define $\mathcal{A}(x)$ the set of all strings that are possible
outcomes after applying the block transformation to $x$, namely if
$\tilde{x} \in \mathcal{A}(x)$ then necessarily
$$b_{l-1}(\tilde{x})=b_1-1, b_l(\tilde{x})=b_2+2,
b_{l+1}(\tilde{x})=b_3-1.$$ However,  not every string $y \in \X_n$
such that $b_{l-1}(y)=b_1-1, b_l(y)=b_2+2$ and $b_{l+1}(y)=b_3-1$
belongs to $\mathcal{A}(x)$. By using \eqref{linearTZ}, it is
straightforward to see that triple
$(b_{l-1}(\tilde{x}),b_l(\tilde{x}),b_{l+1}(\tilde{x}))$ corresponds
to the triple $(u+4,v)$. Since the block-transformation picks up
blocks uniformly, then after applying it to $x$, every element of
$\mathcal{A}(x)$ has the same probability to occur. Formally,
\begin{equation}\label{blocktrans}
P(\tilde{X} = \tilde{x} | X=x) = \left\{  \begin{array}{rl} \theta &\mbox{if $\tilde{x} \in \mathcal{A}(x)$} \\ 0 & \mbox{otherwise} \end{array}\right.
\end{equation}
where $\theta \in (0,1]$ is a constant which depends on $x$ only
through $(b_1,b_2,b_3)$ (or equivalently through $(u,v)$). The
block-transformation only changes blocks of length $l-1$ and $l+1$,
so $\theta=1/(b_1 \cdot b_3)$. The last ingredient is to define the
set $\mathcal{B}(\tilde{x}):=\{ x \in \mathcal{X} : \tilde{x} \in
\mathcal{A}(x)\}$. The cardinality of this set is
\[ | \mathcal{B}(\tilde{x})| = 2 { b_l(\tilde{x}) \choose 2} \]
because, after the block-transformation, each block of length $l$
could have eventually came from two previous shorter or longer
blocks, so the ${ b_l(\tilde{x}) \choose 2}$, and can be either of
1's or of 0's, so the 2.

\noindent To end the proof, let us consider $X \sim P_{(u,v)}$.
Then, we get the corresponding triple
$$(b_{l-1}(u,v),b_l(u,v),b_{l+1}(u,v))$$ depending only on $(u,v)$.
To keep the notation light in what follows, let us call
$b_1^*:=b_{l-1}(u,v), b_2^*:=b_l(u,t,v)$ and $b_3^*:=b_{l+1}(u,v)$.
We aim to find $P(\tilde{X}=\tilde{x})=P(\tilde{X} = \tilde{x} |
U=u,T=t,R=r).$ Note that for every $\tilde{x}\in \X_n$, there are
two possibilities: either ${\cal{B}}(\tilde{x})\cap
S(u,v)=\emptyset$ or
$${\cal{B}}({\tilde{x}})\subset S(u,v).$$
The second case holds if and only if
$$b_{l-1}(\tilde{x})=b^*_1-1, b_l(\tilde{x})=b^*_2+2,
b_{l+1}(\tilde{x})=b^*_3-1$$ or, equivalently $ \tilde{x}\in
S(u+4,v)$. In this case, by \eqref{distXtzr} and \eqref{blocktrans}
we have: {\small
\begin{eqnarray*}
P(\tilde{X} = \tilde{x} | U=u,T=t,R=r) &=& \sum_{x \in \mathcal{B}(\tilde{x})} P(\tilde{X}=\tilde{x} | X=x)P(X=x | U=u,T=t,R=r) \\
&=& \sum_{x \in \mathcal{B}(\tilde{x})} \frac{1}{b_1^* b_3^*}P(X=x | U=t,T=t,R=r) \\
&=& \frac{1}{2 b_1^* b_3^*} { t \choose b_1^* \; b_2^*\;b_3^*}^{-1} |\mathcal{B}(\tilde{x})| \\
&=& \frac{1}{2 b_1^* b_3^*} { t \choose b_1^* \; b_2^*\;b_3^*}^{-1} 2\, { b_2^*+2 \choose 2 } \\
&=& {1\over 2}\, { t \choose b_1^*-1 \;\; b_2^*+2\;\;b_3^*-1}^{-1}\\
&=& P^x_{(u+4,v)}(\tilde{x}).
\end{eqnarray*}}
\end{proof}

\paragraph{Assumption A2.} This assumption means that in the worse case, the length of the LCS decreases in $A$ units after the block-transformation. Let $z=(x,y)$ be a realization of $Z$, $w_{l-1}$ be the block of length $l-1$ and $w_{l+1}$ be the block of length $l+1$ in $x$ that $\mathcal{R}$ has chosen. Note that the decrease in the length of the LCS comes from the following fact: if in every optimal alignment producing the LCS of $x$ and $y$ all the bits of $w_{l+1}$ are aligned, then it is clear that deleting one bit of $w_{l+1}$ will decrease the length of the LCS of $x$ and $y$ only by 1. Later, by adding a new bit in $w_{l-1}$, we cannot get an even lower length of the LCS of $x$ and $y$ (in the worst case, we stay the same). Therefore, $A=1$.
\begin{example}
Take $l=2$, $x=11100010101101$ and $y=11101100101000$. Then $L_{14}(x,y)=10$ could be represented by the alignment
\[\begin{array}{c|c|c|c|c|c|c|c|c|c|c|c|c|c|c|c|c|c|c|c|c}
x &&1&1&1&0&0&0&-&1&-&0&1&0&1&1&0&-&-&1&- \\ \hline
y&&1&1&1&0&-&-&1&1&0&0&1&0&1&-&0&0&0&-&0 \\ \hline
L_{14}(x,y)&&1&1&1&0&&&&1&&0&1&0&1&&0&&&
\end{array}\]
Suppose that $\mathcal{R}$ deletes from the block $w_3=111$ (first block, from left to right) one symbol. The minimum gain for the length of an LCS is when $\mathcal{R}$ adds the extra symbol either to the fifth block in $x$ of length one (from the left to the right):
\[\begin{array}{c|c|c|c|c|c|c|c|c|c|c|c|c|c|c|c|c|c|c|c|c|c}
x &&1&1&-&0&0&0&-&1&-&0&1&1&0&1&1&0&-&-&1&- \\ \hline
y&&1&1&1&0&-&-&1&1&0&0&1&-&0&1&-&0&0&0&-&0 \\ \hline
L_{14}(\tilde{x},y)&&1&1&&0&&&&1&&0&1&&0&1&&0&&&
\end{array}\]
or to the sixth block in $x$ of length one (from the left to the right):
\[\begin{array}{c|c|c|c|c|c|c|c|c|c|c|c|c|c|c|c|c|c|c|c|c|c}
x &&1&1&-&0&0&0&-&1&-&0&1&0&0&1&1&0&-&-&1&- \\ \hline
y&&1&1&1&0&-&-&1&1&0&0&1&-&0&1&-&0&0&0&-&0 \\ \hline
L_{14}(\tilde{x},y)&&1&1&&0&&&&1&&0&1&&0&1&&0&&&
\end{array}\]
In both cases, we get $L_{14}(\tilde{x},y)=9$.
\end{example} 

\vspace{12pt}
\noindent We now state the main theorem of the section, about the linear fluctuations of the length of the LCS in the 3-multinomial block model:
\begin{theorem} \label{LCSthm}
Assume that the block-transformation satisfies ${\bf
A 1}$. Then there exists an universal constant $b>0$ and
$n_6<\infty$ so that for every $n>n_6$, it holds
\begin{equation}\label{varbound}
\v[L(Z)]\geq b \cdot n.
\end{equation}
\end{theorem}
\begin{proof} Let us first check the assumptions of Theorem
\ref{gen2}. {\bf A1} holds by hypothesis; our
block-transformation is such that ${\bf A2}$ holds with $A=1$ (as discussed above). Let
now $n_4$ be as in Lemma \ref{lemmaa}. Recall that
$n_0<n_1<n_2<n_3<n_4<\infty$. Hence, for any $n>n_4$,
(\ref{ulboundsbk}), (\ref{varlin}) and (\ref{LCLT}) hold. Moreover,
from Lemma \ref{lemma09}, it holds $P(U\in \U, V\in \V)\geq 0.9.$
The condition (\ref{ulboundsbk}) states that for every $(u,v)\in
\S\cap \{\U \times \V\}$, we have that $b_{l-1}(u,v)\geq 1$ and
$b_{l+1}(u,v)\geq 1$. Hence, the block-transformation is possible,
and Lemma \ref{tilde} now establishes {\bf A3}. Therefore, for every
$n>n_4$, the assumptions of Theorem \ref{gen2} are fulfilled and,
therefore, there exists $n_5>n_4$ so that for for every $n>n_5$, the
assumptions of Theorem \ref{gen1} hold with $\delta={\e_0\over
2}.$\\
We now apply Theorem \ref{gen1}. As just showed, the assumption {\bf
1}) holds for any $n>n_5$; as explained at the beginning of Subsection \ref{subsec:u}, the
assumption {\bf 3)} holds with $k_o=4$. By (\ref{varlin}), $\psi_n=d
n$. By Theorem \ref{gen1}, thus, the lower bound (\ref{varbound})
exists with
$$b={9 \epsilon_o^2 d \over 640}.$$
\end{proof}
\begin{center}
\bf \large Acknowledgments
\end{center}
F.T. would like to thank the Estonian Science Foundation through the Grant nr.\ 9288 and targeted financing project SF0180015s12 for making possible a two weeks research stay at Tartu University visiting J.L. while working in the core of this article, as well as the DFG through the SFB 878 at University of M\"unster for financial support while the research stay. J.L would like to thank the Estonian Science Foundation through the Grant nr.\ 9288 and targeted financing project SF0180015s12 for supporting a short research stay at University of M\"unster while the finishing of this article.
\appendix
\section{Appendix}
\begin{proposition}\label{propCLTnl} Given $\epsilon >0$ there exist a constant $c'(\e)$  not depending on $n$ but on $\epsilon$ and $n_0(\e)< \infty$  such
that for every $n > n_0$
\begin{equation}\label{CLTnl}
P\left(\left| \frac{B_{l-1} -q_1\frac{n}{\mu }}{\sqrt{n}} \right| \le  c', \quad \left| \frac{B_l -q_2 \frac{n}{\mu }}{\sqrt{n}} \right| \le  c',
\quad \left| \frac{B_{l+1} -q_3\frac{n}{\mu}}{\sqrt{n}} \right| \le c' \right) \ge 1-\epsilon.
\end{equation}
\end{proposition}
\begin{proof}
It suffices to show that for every $\e>0$ there exists
$\gamma_i(\e)>0$ $i=1,2,3$ so that
\begin{equation}\label{all}
P\left(\big|\frac{B_{l_i} - q_i\frac{n}{l}}{\sqrt{n}} \big|\le \gamma_i\right) \ge 1-\epsilon,
\end{equation}
for $i=1,2,3$ and $l_1:=l-1$, $l_2:=l$ and $l_3:=l+1$. From
(\ref{all}) the bound (\ref{CLTnl}) trivially follows by taking
$$c'(\e):=\max_i\big\{\gamma_1({\e\over 3}),\gamma_2({\e\over 3}),\gamma_3({\e\over
3})\big\}.$$ Even more, we shall only show the existence of
one-sided bound for $B_l$: for every $\e$, \label{CLTall} there
exists $\gamma(\e)$ so that
\begin{equation}\label{CLTnls}
P\left(\frac{B_l - q_2\frac{n}{\mu}}{\sqrt{n}} \le \gamma(\e)\right)
\ge 1-\epsilon .
\end{equation}
The other side follows from the same arguments. Since in this proof,
we consider $q_2$, only, let for that proof $q:=q_2$. Let
$\alpha,\beta,\gamma$ be positive real numbers and $m\in
\mathbb{N}$. We define the random variables $\xi_i$ and  events
$A(\alpha,m), B(\beta,n)$ and $C(\gamma,n)$ as following:
\begin{eqnarray}
\xi_i &:=& \left\{\begin{array}{cl}1&\mbox{if $W_i=l$} \\ 0& \mbox{otherwise} \end{array}\right. \nonumber \\
A(\alpha,m) &:=& \left\{ \xi_1+\cdots+\xi_m \le q m + \alpha \sqrt{m} \right\} \nonumber \\
B(\beta,n) &:=& \left\{ B_{l-1}+B_l+B_{l+1} \le \frac{n}{\mu }+\beta \sqrt{n}\right\} \nonumber \\
C(\gamma,n) &:=& \left\{ B_l \le 3 {n\over \mu} + \gamma\sqrt{n}\right\}. \nonumber
\end{eqnarray}
Note that for any $m$, the event  $B_{l-1}+B_l+B_{l+1}>m$ implies
that in the sequence $X$, there are more than $m$ blocks. That, in
turn, means that the $m$ first blocks cover less than $n$ bits of
$X_1,X_2,\ldots$ or, equivalently, $W_1+\cdots+W_m \leq  n$.  Let
$$m(n,\beta):={n\over \mu}+\beta \sqrt{n}.$$ Note that
${n\over m}-\mu<0$. Thus
\begin{eqnarray}
P(B^c(\beta,n))&\le& P\left( W_1+\cdots+W_{m} \le n\right) \nonumber\\
&=& P\left( \frac{W_1+\cdots+W_{m}}{m}-\mu \le \frac{n}{m}-\mu\right) \nonumber\\
\mbox{(by \eqref{corAH} with $\mathrm{P}(|W_1 - \mu| \le 2)=1$)}&\le& \exp\left( -\frac{m}{8}\left(\frac{n}{m} - \mu \right)^2\right) \nonumber\\
&=&\exp\left(-\frac{\mu^3\beta^2}{8}\big({1\over 1+{\beta \mu
\over{\sqrt{n}}}}\big)\right). \label{betamagico}
\end{eqnarray}
Now, for any $\e>0$, we can find $\beta_o=\beta(\e)$ so big that
$\exp[-{\mu^3\beta^2_o\over 2\cdot 8}]<{\e\over 2}$. An then, one
can take $n_0(\e)$ so big that ${\mu \beta_o\over \sqrt{n_o}}<1$.
Hence, for any $n>n_o$,
\begin{equation}\label{b}
P(B^c(n,\beta_o))<{\e\over 2}.
\end{equation}
Let
$$\alpha_0:=\sqrt{2q(1-q)\over \e}.$$
Then by \eqref{chevi}, for any $m$
\begin{equation}\label{alphamagico}
P(A^c(\alpha_0,m)) = P\left( \frac{\xi_1 + \cdots +\xi_{m}}{m}-q \ge \frac{\alpha_0}{\sqrt{m}} \right) \le \frac{\epsilon}{2}.
\end{equation}
Finally, if we define
\[ \gamma(\beta,\alpha_0):= q \beta+\alpha_0 \sqrt{\frac{1}{\mu}+\beta}\]
then we have
\[\left.\begin{array}{c}
\xi_1+\dots+\xi_{m(\beta,n)}\le q{m(\beta,n)}+\alpha_0\sqrt{m(\beta,n)} \\
B_{l-1}+B_l+B_{l-1} \le m(\beta,n)
\end{array}\right\} \Rightarrow B_l \le q{m(\beta,n)}+\alpha_0\sqrt{m(\beta,n)}\leq  q \frac{n}{\mu} + \gamma\sqrt{n}.\]
Therefore
\begin{equation}\label{inclusion}
A(\alpha_0,m) \cap B(\beta_o,n) \subseteq C(\gamma,n).
\end{equation}
Now take $n>n_0$, $m_0:=m(\beta_0,n)$ and
$\gamma_0:=\gamma(\beta_0,\alpha_0)$ and use
 \eqref{alphamagico}, \eqref{betamagico}
and \eqref{inclusion} to get
\begin{equation}\label{comineq}
P\left(\frac{B_l - q \frac{n}{\mu }}{\sqrt{n}} >  \gamma_0\right) = P(C^c(\gamma_0,n)) \le P(A^c(\alpha_0,m_0))+P(B^c(\beta_0,n))\le \epsilon.
\end{equation}
\end{proof}
\paragraph{Proof of Lemma \ref{lemma09}.}
By Proposition \ref{propCLTnl}, for any $\epsilon>0$ there exist
$c'=c'(\epsilon)>0$ and $n_0(\e) <\infty$ such that:
{\footnotesize\begin{eqnarray*} P\left(
\left|\frac{B_{l-1}+B_l+B_{l+1} -\frac{n}{\mu}}{\sqrt{n}}\right| >
3c' \right) & \le
& P\left( \left| \frac{B_{l-1}-q_1\frac{n}{\mu}}{\sqrt{n}}\right| + \left| \frac{B_l-q_2\frac{n}{\mu}}{\sqrt{n}}\right| + \left| \frac{B_{l+1}-q_3\frac{n}{\mu}}{\sqrt{n}}\right| > 3c'\right)\\
& \le & \sum_{i=1}   P\left( \left| \frac{B_{l_i}-q_i\frac{n}{\mu}}{\sqrt{n}}\right| > c' \right) \le 3\epsilon\\
P\left( \left|\frac{B_l-B_{l-1}-B_{l+1} -\frac{n}{\mu}(q_2-q_1-q_3)}{\sqrt{n}}\right|> 3c' \right) &\le& P\left( \left| \frac{B_{l-1}-q_1\frac{n}{\mu}}{\sqrt{n}}\right| +
\left| \frac{B_l-q_2\frac{n}{\mu}}{\sqrt{n}}\right| + \left| \frac{B_{l+1}-q_3\frac{n}{\mu}}{\sqrt{n}}\right| > 3c'\right)\\
& \le & \sum_{i=1} P\left( \left| \frac{B_{l_i}-q_i\frac{n}{\mu}}{\sqrt{n}}\right| > c' \right) \le 3\epsilon\\
\end{eqnarray*}}
\noindent for any $n>n_0$, where as before $l_1=l-1,l_2=l$ and
$l_3=l+1$. Then, it directly follows:
\[P\big(U\notin \U^{3c'}, V\notin \V^{3c'}\big) \ge 6\epsilon\]
from where the proof is completed by taking $c:=3c'$ and
$\epsilon=1/60$.\eop

\paragraph{Proof of Lemma \ref{RMpossible}.}
Let $n>n_0$. Take $(u,v) \in \mathcal{S}_n \cap (\mathcal{U}_n
\times \mathcal{V}_n)$ (By lemma \ref{lemma09}, the set
$\mathcal{S}_n \cap (\mathcal{U}_n \times \mathcal{V}_n)$ is not
empty). In particular,
\[t= \frac{n}{\mu} + \sigma_1, \quad u=\frac{n}{\mu}(q_2-q_1-q_3)+\sigma_2,\quad r \in \{1,\dots,\l+1\},\]
for $\sigma_1,\sigma_2 \in [-c\sqrt{n},c\sqrt{n}\,]$. Let us start
by showing that $b_{l_i}(u,t,r) \ge 1$. Recall $\mu=l+q_3-q_1$. From
\eqref{linearTZ} and $1 \leq r <l+1$ we get
\begin{eqnarray*}
b_{l-1}(u,t,r) &=& q_1\frac{n}{\mu} + \frac{2\sigma_1l+(\sigma_1-\sigma_2)+2r }{4} \\
&\ge& q_1\frac{n}{\mu} - \frac{c(l+1)}{2}\sqrt{n} \ge 1
\end{eqnarray*}
provided $n > n_{1,1}(c)$. Also, by using \eqref{linearTZ} and $1 <r
<l+1$ we get
\begin{eqnarray*}
b_l(u,t,r) = q_2\frac{n}{\mu} + \frac{\sigma_1+\sigma_2}{2} &\ge&
q_2\frac{n}{\mu} - c\sqrt{n} \ge 1
\end{eqnarray*}
provided $n > n_{1,2}(c)$. Finally, also by using \eqref{linearTZ}
and $1 \leq r <l+1$ we get
\begin{eqnarray*}
b_{l+1}(u,t,r) &=& q_3\frac{n}{\mu} - \frac{2\sigma_1l+(\sigma_2-\sigma_1)+2r }{4} \\
&\ge& q_3\frac{n}{\mu} - \frac{c(l+1)}{2}\sqrt{n}-\frac{l-2}{2} \ge
1
\end{eqnarray*}
provided $n > n_{1,3}(c).$ So, for having simultaneously the three
lower bounds we need to take $n>n_1^1:=\max\{
n_{1,1},n_{1,2},n_{1,3}\}$.

\noindent For the absolute value bounds, we proceed in the same way:
\begin{eqnarray*}
\left| b_{l-1}(u,t,r) -q_1\frac{n}{\mu} \right|= \left| \frac{2\sigma_1l+(\sigma_1-\sigma_2)+2r }{4} \right| &\le& \frac{c(l+2)}{2} \sqrt{n},\quad\mbox{for $n \ge (l-2)^2 / c^2$} \\
\left| b_l(u,t,r) - q_2\frac{n}{\mu} \right| = \left|\frac{\sigma_1+\sigma_2}{2}\right| &\le& c\sqrt{n},\quad\mbox{for $n \ge 1$} \\
\left|b_{l+1}(u,t,r) - q_3\frac{n}{\mu}\right| =\left| -
\frac{2\sigma_1l+(\sigma_2-\sigma_1)+2r }{4}\right| &\le&
\frac{c(l+1)}{2}\sqrt{n},\quad\mbox{for $n \ge (l-2)^2 / c^2$.}
\end{eqnarray*}
Thus, all the above three upper bounds hold for
$n_1^2:=\max\{1,(l-2) / c\}$. In order to obtain \eqref{ulboundsbk},
it is enough to take $n_1=\max\{ n_1^1,n_1^2\}$ and the universal
constant $\alpha:=c(l+2)/2$ which does not depend on $(u,t,r)$.
Without loss of generality, we can take $n_1>n_0$. \eop


\paragraph{Proof of Lemma \ref{quotientTZR}.}
Let $n>n_1$ and  consider $(u,v) \in \mathcal{S}_n \cap
(\mathcal{U}_n\times \mathcal{V}_n)$. Let
\[(b_{l-1}(u,v),b_l(u,v),b_{l+1}(u,v))\]
be the solution of
\eqref{linearTZ}. We have already seen that
\begin{eqnarray}
b_{l-1}(u+4,v) &=& b_{l-1}(u,v)-1 \nonumber \\
b_l(u+4,v) &=& b_l(u,v)+2 \nonumber \\
b_{l+1}(u+4,v) &=& b_{l+1}(u,v)-1. \nonumber
\end{eqnarray}
Therefore, by using these relations and \eqref{TZRdist} (formula of
the multinomial distribution for $(U,T,R)$) we get:
\begin{equation}\label{cuocient}
\frac{P(U_{(v)}=u+4|U_{(v)}\in \U)}{P(U_{(v)}=u|U_{(v)}\in \U)} =
\frac{b_{l-1}(u,v) b_{l+1}(u,v)}{(b_l(u,v)+1)(b_l(u,v)+2)}
\cdot\frac{q_2^2}{q_1 q_3}.
\end{equation}
By using Lemma \ref{RMpossible}, there exist an universal constant
$\alpha>0$  such that:
{\small \begin{eqnarray}
\frac{b_{l-1}(u,v) b_{l+1}(u,v)}{(b_l(u,v)+1)(b_l(u,v)+2)} \cdot \frac{q_2^2}{q_1q_3} &\ge&\frac{\left( q_1\frac{n}{\mu}-\alpha\sqrt{n}\right)\left( q_3\frac{n}{\mu}-\alpha\sqrt{n}\right)}{\left( 1+q_2\frac{n}{\mu}+\alpha\sqrt{n}\right)\left( 2+q_2\frac{n}{\mu}+\alpha\sqrt{n}\right)}\cdot \frac{q_2^2}{q_1q_3} \label{masterineA} \\
\frac{b_{l-1}(u,v) b_{l+1}(u,v)}{(b_l(u,v)+1)(b_l(u,v)+2)} \cdot
\frac{q_2^2}{q_1q_3} &\le& \frac{\left(
q_1\frac{n}{\mu}+\alpha\sqrt{n}\right)\left(
q_3\frac{n}{\mu}+\alpha\sqrt{n}\right)}{\left(
1+q_2\frac{n}{\mu}-\alpha\sqrt{n}\right)\left(
2+q_2\frac{n}{\mu}-\alpha\sqrt{n}\right)}\cdot
\frac{q_2^2}{q_1q_3}\label{masterineB}
\end{eqnarray}}
for every $n>n_1$.  Recall the inequalities (\ref{log}). Let us
start by looking at \eqref{masterineA}:
{\small \begin{eqnarray}
\frac{\left( q_1\frac{n}{\mu}-\alpha\sqrt{n}\right)\left( q_3\frac{n}{\mu}-\alpha\sqrt{n}\right)}{\left( 1+q_2\frac{n}{\mu}+\alpha\sqrt{n}\right)\left( 2+q_2\frac{n}{\mu}+\alpha\sqrt{n}\right)}\cdot \frac{q_2^2}{q_1q_3} &\ge& \frac{\left( 1-\frac{\mu\alpha/q_1}{\sqrt{n}}\right)\left( 1-\frac{\mu\alpha/q_3}{\sqrt{n}}\right)}{\left( 1+\frac{2\mu\alpha/q_2}{\sqrt{n}}\right)^2} \nonumber\\
&=& \exp\left[ \ln\left( 1-\frac{\mu\alpha/q_1}{\sqrt{n}}\right)+\ln\left( 1-\frac{\mu\alpha/q_3}{\sqrt{n}}\right)\right. \nonumber\\
&&\left.-2\ln\left(1+\frac{2\mu\alpha/q_2}{\sqrt{n}}\right)\right] \nonumber\\
\left(\mbox{from \eqref{log} for every $n>\max\{ \frac{\mu^2\alpha^2}{q_1^2},\frac{\mu^2\alpha^2}{q_3^2}\}$}\right)& \ge& \exp\left[ -\frac{\mu\alpha}{2\sqrt{n}}\left( \frac{3}{q_1} + \frac{8}{q_2}+\frac{3}{q_3}\right)\right]\nonumber \\
&\ge&1-\frac{\mu\alpha}{2}\left( \frac{3}{q_1} +
\frac{8}{q_2}+\frac{3}{q_3}\right) \frac{1}{\sqrt{n}}.\label{ine1}
\end{eqnarray}}
Next, let us fix an arbitrary $\epsilon>0$. Then, there exists
$n_{2,1} < \infty$ such that the rest of the Taylor's expansion of
the function $f(x)=e^{-x}$ at $\xi:=\frac{\mu\alpha}{\sqrt{n}}\left(
\frac{1}{q_1}+\frac{6}{q_2}+\frac{1}{q_3}\right)$ satisfies:
\begin{equation}\label{restTaylor}
R(\xi):=\left|\frac{f''(\xi)}{2}\right|\xi^2=
\frac{\mu^2\alpha^2}{n}\left(\frac{1}{q_1}+\frac{6}{q_2}+\frac{1}{q_3}\right)^2\exp\left(-\frac{\mu\alpha}{\sqrt{n}}\left(
\frac{1}{q_1}+\frac{6}{q_2}+\frac{1}{q_3}\right)\right) \le \epsilon
\end{equation}
for every $n > n_{2,1}$. Note that $n_{2,1}$ does not depend on
$(u,v)$ but only on known fixed constants. Now, let us look at
\eqref{masterineB}: {\small \begin{eqnarray}
\frac{\left( q_1\frac{n}{\mu}+\alpha\sqrt{n}\right)\left( q_3\frac{n}{\mu}+\alpha\sqrt{n}\right)}{\left( 1+q_2\frac{n}{\mu}-\alpha\sqrt{n}\right)\left( 2+q_2\frac{n}{\mu}-\alpha\sqrt{n}\right)}\cdot \frac{q_2^2}{q_1q_3} &\le& \frac{\left( 1+\frac{\mu\alpha/q_1}{\sqrt{n}}\right)\left( 1+\frac{\mu\alpha/q_3}{\sqrt{n}}\right)}{\left( 1-\frac{2\mu\alpha/q_2}{\sqrt{n}}\right)^2} \nonumber\\
&=& \exp\left[ \ln\left(1+\frac{\mu\alpha/q_1}{\sqrt{n}}\right)+\ln\left(1+\frac{\mu\alpha/q_3}{\sqrt{n}}\right) \right.\nonumber\\
&&\left.-2\ln\left(1-\frac{2\mu\alpha/q_2}{\sqrt{n}}\right)\right] \nonumber\\
\left(\mbox{from \eqref{log} for every $n>\frac{16\mu^2\alpha^2}{q_2^2}$}\right) &\le&\exp\left[\frac{\mu\alpha}{\sqrt{n}}\left( \frac{1}{q_1}+\frac{6}{q_2}+\frac{1}{q_3}\right)\right] \nonumber\\
&\le&  1+ \mu\alpha\left( \frac{1}{q_1}+\frac{6}{q_2}+\frac{1}{q_3}\right)\frac{1}{\sqrt{n}}+|R(\xi)| \nonumber \\
\mbox{from \eqref{restTaylor}}&\le& 1+ \mu\alpha\left(
\frac{1}{q_1}+\frac{6}{q_2}+\frac{1}{q_3}\right)\frac{1}{\sqrt{n}}+\epsilon
\label{ine2}
\end{eqnarray}}
for every
$n>n_{2,2}:=\max\{n_{2,1},\frac{16\mu^2\alpha^2}{q_2^2}\}$. Finally,
from \eqref{ine1} and \eqref{ine2} we have:
\[ 1-\frac{\mu\alpha}{2}\left( \frac{3}{q_1} + \frac{8}{q_2}+\frac{3}{q_3}\right) \frac{1}{\sqrt{n}} \le
\frac{P(U_{(v)}=u+4|U_{(v)}\in \U)}{P(U_{(v)}=u|U_{(v)}\in \U)} \le
1+ \mu\alpha\left(
\frac{1}{q_1}+\frac{6}{q_2}+\frac{1}{q_3}\right)\frac{1}{\sqrt{n}}+\epsilon\
\] for any arbitrary $\epsilon>0$ .
>From this last inequality, we can find a constant $K>0$ not
depending on $n$ neither on $(u,v)$ and $n_2$ bigger than $n_{2,2}$
and $n_1$ such that \eqref{BinScope} holds for every $n>n_2$.\eop
\paragraph{Proof of Lemma \ref{lemmaa}.}
The proof is based on the  Corollary \ref{multiclt}. Define
\begin{equation}\label{beta} \beta:=(\alpha+c)\sqrt{2\mu},
\end{equation}
 where $\alpha$ is as in
(\ref{ulboundsbk}) and choose $n_4>n_3$ so big that simultaneously
${n_4\over \mu}- c \sqrt{n}>m_o(\beta)$ and $\sqrt{n_4}\geq 2c\mu$.
Here $m_o(\beta)$ is as in Corollary \ref{multiclt}. From these
inequalities, it follows that  whenever $n>n_4$, then
\begin{equation}\label{n4}
{n\over \mu}-c\sqrt{n}\geq \max\big\{{1\over 2 \mu}n,m_o\big\}.
\end{equation}
Take now $n>n_4$ and $(u,t,r) \in \mathcal{S}_n \cap (\U \times
\V)$. By (\ref{n4}),
\begin{equation}\label{n42}
t\geq {n\over \mu}-c \sqrt{n}>m_o(\beta),\quad 2\mu t\geq
n.\end{equation} Use now (\ref{ulboundsbk}) and the definition of $\V$ to see  that for every $i=1,2,3$,
\begin{align*}
|b_{l_i}-t q_i|\leq |b_{l_i}- q_i\big({n\over \mu}\big)|+ q_i|{n\over \mu}-t|\leq \alpha \sqrt{n}+q_ic\sqrt{n}\leq (\alpha +c)\sqrt{n}\leq (\alpha+c)\sqrt{2\mu}\sqrt{t},
\end{align*}
where the last inequality follows from (\ref{n42}) and
$b_{l_i}=b_{l_i}(u,t,r)$. Apply (\ref{multi}) with $\beta$ as in
(\ref{beta}), $m=t$, $p_1=q_1$, $p_2=q_2$, $p_3=q_3$ and
$i=b_{l_1}$, $j=b_{l_2}$. Then by (\ref{TZRdist})
\begin{equation}\label{multi2}
P(U=u,T=t,R=r) ={ t \choose b_{l_1}\;b_{l_2}\;b_{l_3}}q_1^{b_{l_1}}q_2^{b_{l_2}}q_3^{b_{l_3}}p(r)\geq {p(r)\over b(\beta)n}\geq {q_3\over b(\beta)n},
\end{equation}
where the last inequality comes from the fact that $p(r)\geq q_3$.
Thus Lemma \ref{lemmaa} is proven with
$$a={b(\beta)\over q_3}.$$\eop

\newpage
\bibliographystyle{alpha}

\end{document}